\documentclass[reqno]{amsart}

\usepackage{lmodern}
\usepackage[T1]{fontenc}
\usepackage[utf8]{inputenc}
\usepackage[english]{babel}
\usepackage{microtype} 

\usepackage{amsmath,amssymb,amsthm,mathrsfs,cases,
latexsym,mathtools,mathdots,booktabs,colortbl,xparse,enumitem,tikz,bm,url}

\usetikzlibrary{arrows,positioning,shapes.geometric}
\usepackage[pdftex]{hyperref}
\usepackage[capitalize,noabbrev]{cleveref}

\newtheorem{theorem}{Theorem}
\newtheorem{proposition}[theorem]{Proposition}
\newtheorem{lemma}[theorem]{Lemma}
\newtheorem{corollary}[theorem]{Corollary}
\newtheorem{conjecture}[theorem]{Conjecture}

\theoremstyle{definition}
\newtheorem{definition}[theorem]{Definition}
\newtheorem{example}[theorem]{Example}
\newtheorem{remark}[theorem]{Remark}

\newtheorem{problem}[theorem]{Problem}

\definecolor{lightblue}{rgb}{0.8,0.8,1.0}
\definecolor{lightgreen}{rgb}{0.8,1.0,0.8}
\definecolor{pBlue}{RGB}{86,139,190}
\definecolor{pCyan}{RGB}{149,186,201}
\definecolor{pSand}{RGB}{184,166,121}
\definecolor{pAlgae}{RGB}{87,115,135}
\definecolor{pSkin}{RGB}{236,216,167}
\definecolor{pGray}{RGB}{156,175,156}
\definecolor{pPink}{RGB}{215,114,127}
\definecolor{pOrange}{RGB}{211,153,80}

\newcommand{\stirlingOne}[2]{{S_1(#1,#2)}}

\newcommand{\defin}[1]{%
\relax\ifmmode%
\textcolor{blue}{#1}%
\else\textcolor{blue}{\emph{#1}}%
\fi%
}

\newcommand{\oeis}[1]{\href{http://oeis.org/#1}{#1}}
\newcommand{\thsup}{\textnormal{th}}

\newcommand{\setN}{\mathbb{N}}

\newcommand{\xvec}{\mathbf{x}}
\newcommand{\yvec}{\mathbf{y}}
\newcommand{\wvec}{\mathbf{w}}
\newcommand{\vvec}{\mathbf{v}}

\newcommand{\symS}{\mathfrak{S}}
\newcommand{\dArr}{\mathfrak{D}}
\newcommand{\bdArr}{\mathfrak{BD}}

\newcommand{\SEF}{\mathcal{F}}  
\newcommand{\DSEF}{\mathcal{DF}}

\newcommand{\frether}{\hat{\Psi}}
\newcommand{\fretherD}{\Psi}

\NewDocumentCommand\eqCond{o}{%
\IfNoValueTF{#1}{%
{${\circledast}$}%
}{%
{${\circledast_{#1}}$}%
}}

\NewDocumentCommand\caseOne{o}{%
\IfNoValueTF{#1}{%
{\ensuremath{\heartsuit}}%
}{%
{\ensuremath{\heartsuit_{#1}}}%
}}

\NewDocumentCommand\caseTwo{o}{%
\IfNoValueTF{#1}{%
{\ensuremath{\spadesuit}}%
}{%
{\ensuremath{\spadesuit_{#1}}}%
}}

\NewDocumentCommand\caseThree{o}{%
\IfNoValueTF{#1}{%
{\ensuremath{\diamondsuit}}%
}{%
{\ensuremath{\diamondsuit_{#1}}}%
}}

\NewDocumentCommand\caseFour{o}{%
\IfNoValueTF{#1}{%
{\ensuremath{\clubsuit}}%
}{%
{\ensuremath{\clubsuit_{#1}}}%
}}

\newcommand{\img}{\mathbf{m}}

\DeclareMathOperator{\inv}{inv}
\DeclareMathOperator{\sgn}{sgn}
\DeclareMathOperator{\type}{type}

\DeclareMathOperator{\cycle}{\mathfrak{c}}

\DeclareMathOperator{\fix}{fix}
\DeclareMathOperator{\FIX}{FIX}

\DeclareMathOperator{\exc}{exc}
\DeclareMathOperator{\EXCi}{EXCi}
\DeclareMathOperator{\EXCv}{EXCv}

\DeclareMathOperator{\rlm}{rlm}
\DeclareMathOperator{\RLMv}{RLMv}
\DeclareMathOperator{\RLMi}{RLMi}

\DeclareMathOperator{\aexc}{sae} 
\DeclareMathOperator{\supp}{IM}

\DeclareMathOperator{\fixMap}{\mathtt{fix}}
\DeclareMathOperator{\unfixMap}{\mathtt{unfix}}
\DeclareMathOperator{\sefToPerm}{\mathtt{sefToPerm}}
\DeclareMathOperator{\flipMap}{\mathtt{flip}}
\DeclareMathOperator{\zetaMap}{\zeta}

\title[An involution on derangements preserving EXC and RL-minima]{An involution on derangements preserving excedances and right-to-left minima}

\author[Per Alexandersson]{Per Alexandersson}
\address{Department of Mathematics, Stockholm University, SE-106 91 Stockholm, Sweden}
\email{per.w.alexandersson@gmail.com}

\author[Frether Getachew]{Frether Getachew}
\address{Department of Mathematics, Addis Ababa University, Ethiopia}
\email{frigetach@gmail.com}

\begin{document}


\begin{abstract}
 We give a bijective proof of a result by R.~Mantaci and F.~Rakotondrajao from 2003,
 regarding even and odd derangements with a fixed number of excedances.
 We refine their result by also considering the set of right-to-left minima.
\end{abstract}

\maketitle 

\section{Introduction and preliminaries}

Let $\defin{\symS_n}$
be the symmetric group acting on the set
$\defin{[n]} \coloneqq \{1,2,\dotsc,n\}$. An integer $i\in[n]$ is said to be a \defin{fixed point} 
of a permutation $\pi\in\symS_n$ if $\pi(i)=i$. The set of fixed points of $\pi$ is denoted by $\defin{\FIX(\pi)}$
and we set $\defin{\fix(\pi)} \coloneqq |\FIX(\pi)|$.
Recall that the set of \defin{derangements} is defined as
$
\defin{\dArr_n} \coloneqq \{ \pi \in \symS_n : \fix(\pi) = 0 \}.
$

An \defin{inversion} in a permutation 
$\pi$ is a pair $(i,j)$, for 
$1\leq i < j \leq n$, such that 
$\pi(i)>\pi(j)$. The parity of the number of
inversions, $\defin{\inv(\pi)}$, in a permutation $\pi$ determines the 
parity of the permutation; $\pi$ is \defin{even} if $\inv(\pi)$ is even, otherwise $\pi$ is called an \defin{odd} permutation. 
The \defin{sign} of $\pi$, $\defin{\sgn(\pi)}$ is defined as $(-1)^{\inv(\pi)}$.
The set of even 
permutations in $\symS_n$ is denoted
$\defin{\symS^e_n}$ and the set of odd 
permutations is $\defin{\symS^o_n}$. Let 
$\defin{\dArr^e_n}$ and $\defin{\dArr^o_n}$ be 
the sets of even and odd 
derangements, respectively, in $\dArr_n$. 
Whenever $S = \{s_1,\dotsc,s_m\}$ is a finite set 
of positive integers,
we shall let $\defin{\xvec_S}$ denote the product $x_{s_1}x_{s_2}\dotsm x_{s_m}$.
By definition, $\defin{\xvec_\emptyset} \coloneqq 1$.

In order to state our results, we need to 
recall some standard notions and terminology.
For any function $g:[n]\longrightarrow [n]$, let 
the set of \defin{excedances}, \defin{excedance values},
\defin{right-to-left minima indices}, \defin{right-to-left minima values}, and \defin{the number of inversions}
respectively, be defined as 
\begin{align*}
\defin{\EXCi(g)} &\coloneqq \{ j \in [n] : g(j)>j \}, \\
\defin{\EXCv(g)} &\coloneqq \{ g(j) : j \in \EXCi(g) \}, \\
\defin{\RLMi(g)} &\coloneqq  \{i\in [n]: g(i)<g(j) \text{ for all } j\in \{i+1,\dotsc,n\} \},\\
\defin{\RLMv(g)}  &\coloneqq  \{ g(i) : i \in \RLMi(g)\}, \\
\defin{\inv(g)} &\coloneqq |\{ (i,j) : 1\leq i < j \leq n \text{ such that } g(i)>g(j) \}|.
\end{align*}
Moreover, we denote $\defin{\exc(g)} \coloneqq |\EXCi(g)|$ and $\defin{\rlm(g)} \coloneqq |\RLMi(g)| = |\RLMv(g)|$. Note that,  $|\EXCv(\sigma)|=|\EXCi(\sigma)|=\exc(\sigma)$, for any $\sigma\in \symS_n$, and indices which are not excedances are called \defin{anti-excedances}. Below we show three permutations in $\symS_7$.
The first permutation has $3$ and $6$ as 
fixed-points, so it is not a derangement,
while the remaining two are.
\begin{center}
	\begin{tabular}{l c llll}
		Permutation, $\pi$ & $\inv(\pi)$ & $\EXCi(\pi)$ &  $\RLMi(\pi)$ & $\RLMv(\pi)$  \\
		\midrule 
		2135764 & 5 & \{1,4,5\} & \{2,3,7\} & \{1,3,4\}  \\
		2153746 & 5 & \{1,3,5\} & \{2,4,6,7\} & \{1,3,4,6\}  \\
		6713245 & 11 & \{1,2\} & \{3,5,6,7\} & \{1,2,4,5\} 
	\end{tabular}
\end{center}

The right-to-left minima statistic and the excedance statistic 
behave quite differently. One can see that
\[
\sum_{\pi \in \symS_n} t^{\rlm(\pi)} = 
\sum_{\pi \in \symS_n} t^{\cycle(\pi)} =
\sum_{k=1}^n \stirlingOne{n}{k} t^k 
\]
where $\defin{\cycle(\pi)}$ is the number of cycles in cycle representation of $\pi$ 
and $\stirlingOne{n}{k}$ is the unsigned Stirling number of the first kind,
see \oeis{A008275}.
However,
\[
\sum_{\pi \in \symS_n} t^{\exc(\pi)} = 
\sum_{k=1}^{n} A_{n,k} t^{k-1} 
\]
where $A_{n,k}$ denote the Eulerian numbers, \oeis{A008292}.

\bigskip
It was shown\footnote{Their proof uses a recursion, rather than an explicit involution.}
by R.~Mantaci and F.~Rakotondrajao~\cite[Proposition 4.3]{MantaciRakotondrajao2003},
that for every $n\geq 1$ and $1 \leq k \leq n-1$,
\begin{equation}\label{eq:mainEquation}
 |\{ \pi \in \dArr^e_n  : \exc(\pi) = k \}|
 -
 |\{ \pi \in \dArr^o_n  : \exc(\pi)=k \}|
 =(-1)^{n-1}.
\end{equation}
This refines a result by Chapman, stating that
$|\dArr^e_n|-|\dArr^o_n|=(-1)^{n-1}(n-1)$, see \cite{Chapman2001}. 

\medskip
In this paper, we provide a proof for a refinement of \eqref{eq:mainEquation}, namely
\begin{equation}\label{eq:mainResult}
 \sum_{\pi \in \dArr_n} (-1)^{\inv(\pi)} \!
 \left( \hspace{5mm}
 \prod_{\mathclap{j\in \RLMv(\pi)}}  x_j 
 \hspace{3mm}
 \right)
 \!\!
 \left( \hspace{5mm}
 \prod_{\mathclap{j\in \EXCv(\pi)}}  y_{j}
 \hspace{3mm}
 \right)
 = 
 (-1)^{n-1} \sum_{j=1}^{n-1} x_1 \!\dotsm \! x_j y_{j+1}  \!\dotsm  \! y_{n},
\end{equation}
in \cref{sec:involution}.
We prove \eqref{eq:mainResult} by
exhibiting a bijection $\frether: \dArr_n \to \dArr_n$
with exactly $(n-1)$ fixed-elements, where $\frether$ acts as a sign-reversing involution outside the set of fixed-elements.
The bijection preserves
the excedance value and right-to-left minima permutation statistics,
which gives the desired result.
Moreover, \eqref{eq:mainResult} allows us to deduce that
\begin{equation*}\label{eq:mainResult2}
\sum_{\pi \in \dArr_n} (-1)^{\inv(\pi)} \!
 \left( \hspace{5mm}
 \prod_{\mathclap{j\in \RLMi(\pi)}}  x_j 
 \hspace{3mm}
 \right)
 \!\!
 \left( \hspace{5mm}
 \prod_{\mathclap{j\in \EXCi(\pi)}}  y_{j}
 \hspace{3mm}
 \right)
 = 
 (-1)^{n-1} \sum_{j=1}^{n-1} y_1 \!\dotsm \! y_j x_{j+1}  \!\dotsm  \! x_{n},
\end{equation*}
where we now consider indices instead of values.

\medskip 


We include an alternative proof, using generating functions, of the $x_1=x_2=\dotsb = x_n=1$
case of \eqref{eq:mainResult}
in \cref{sec:genFunc}. In addition, \cref{sec:genFunc} includes the proof of the following result:

\begin{equation*}\label{eq:rlminSumIntro}
 \sum_{\substack{\pi \in \symS_n   }} (-1)^{\inv(\pi)} 
 \left(\prod_{j \in \RLMv(\pi)} x_j \right)
 = 
	\big(\prod_{\substack{i \in [n] \\ i \text{ odd}}} x_i\big)
	\big(\prod_{\substack{j \in [n] \\ j \text{ even}}} (x_j-1)\big).
\end{equation*}

\section{Subexcedant functions}

The involution we shall construct is not performed directly on permutations,
but rather on so-called \emph{subexcedant functions} which are in bijection with permutations.
Our main reference is \cite{MantaciRakotondrajao2001},
where several fundamental properties are proved.
\begin{definition}\label{subex}
	A \defin{subexcedant function}  $f$ on $[n]$ is a map $f : [n] \longrightarrow [n]$
	such that
	\begin{equation*}
		1\leq f(i)\leq i \text{ for all } 1\leq i\leq n.
	\end{equation*}
	We let $\defin{\SEF_n}$ denote the set of all 
	subexcedant functions on $[n]$. 
	The \defin{image} of $f \in \SEF_n$ is defined as
	$\defin{\supp(f)} \coloneqq \{ f(i) : i \in [n] \}$.
\end{definition}
We write subexcedant functions as words, $f(1)f(2)\dotsc f(n)$.
For example, the subexcedant function $f=112352$ has $\supp(f)=\{1,2,3,5\}$.

From each subexcedant function $f \in \SEF_{n-1}$,
one can obtain $n$ distinct subexcedant functions in $\SEF_{n}$ by appending
any integer $i\in [n]$ at the end of the word representing $f$.
Hence, the cardinality of $\mathcal{F}_n$ is $n!$.
There is a bijection $\defin{\sefToPerm}: \SEF_n \longrightarrow \symS_n$,
described in \cite{MantaciRakotondrajao2001}, 
which is defined as the following composition (using cycle notation for permutations):
\begin{equation*}\label{eq:sefToPerm}
  \sefToPerm(f) \coloneqq (n\,\,f(n))  \dotsm (2\,\,f(2)) (1\,\,f(1)).
\end{equation*}

\begin{example}\label{ex:sefToPerm}
Let $f = 112435487 \in \SEF_9$.
The permutation $\sigma=\sefToPerm(f)$ is
 \begin{align*}
 	\sigma &=(9\,\,7)(8)(7\,\,4)(6\,\,5)(5\,\,3)(4)(3\,\,2)(2\,\,1)(1)\\
 	&=(1\,\,6\,\,5\,\,3\,\,2)(4\,\,9\,\,7)(8)\\
 	&=6\,1\,2\,9\,3\,5\,4\,8\,7.
 \end{align*}
\end{example}

For $\sigma \in \symS_n$ and $j \in [n]$,
it is fairly straightforward to see that we can compute
the $j^\thsup$ entry of $\sefToPerm^{-1}(\sigma)$
via the recursive formula
\begin{equation}\label{eq:permToSefRecursion}
 \sefToPerm^{-1}(\sigma)_j \coloneqq 
\begin{cases}
 \sigma(n) \text{ if $j=n$,} \\
 \sefToPerm^{-1}\big( \left(n\,\,\sigma(n)\right)\circ\sigma \big)_j &\text{ otherwise}.
\end{cases}
\end{equation}
Note that 
\begin{equation}\label{eq:permProjection}
 \sigma' \coloneqq \left(n\,\,\sigma(n)\right)\circ\sigma
\end{equation}
is the result after interchanging $n$ and the image of $n$ in $\sigma$.
Therefore, $\sigma'(n)= n$ and, by a slight abuse of notation, $\sigma'$ can be considered as a permutation in $\symS_{n-1}$.
Hence, the definition above is well-defined, and for simplicity,
we use the shorthand $\defin{f_\sigma} \coloneqq \sefToPerm^{-1}(\sigma)$.

\begin{example}\label{ex:permToSef}
We shall now show how to invert the calculation in \cref{ex:sefToPerm}.
We start with the permutation $\sigma^9=\binom{1\,2\,3\,4\,5\,6\,7\,8\,9}{6\,1\,2\,9\,3\,5\,4\,8\,7}$ using two line notation,
and for $i \geq 1$ we let $\sigma^{i-1} \in \symS_{i-1}$ be given by
\[
  \sigma^{i-1} \coloneqq (i\,\,\sigma^i(i))\circ\sigma^{i}  ,
\]
where we use the observation in \eqref{eq:permProjection}.
Combining this recursion with \eqref{eq:permToSefRecursion},
we have
\begin{equation*}
\begingroup
\renewcommand*{\arraystretch}{1.2}
\begin{array}{lr}
	\sigma^9=\binom{1\,2\,3\,4\,5\,6\,7\,8\,9}{6\,1\,2\,9\,3\,5\,4\,8\,\mathbf 7}  & f_{\sigma}(9)=7 \\
	\sigma^8=\binom{1\,2\,3\,4\,5\,6\,7\,8}{6\,1\,2\,7\,3\,5\,4\,\mathbf 8 } & f_{\sigma}(8)=8 \\
	\sigma^7=\binom{1\,2\,3\,4\,5\,6\,7}{6\,1\,2\,7\,3\,5\, \mathbf 4} & f_{\sigma}(7)=4 \\
	\sigma^6=\binom{1\,2\,3\,4\,5\,6}{6\,1\,2\,4\,3\,\mathbf  5} & f_{\sigma}(6)=5 \\
	\sigma^5=\binom{1\,2\,3\,4\,5}{5\,1\,2\,4\,\mathbf  3} & f_{\sigma}(5)=3 \\
	\sigma^4=\binom{1\,2\,3\,4}{3\,1\,2\,\mathbf  4} & f_{\sigma}(4)=4 \\
	\sigma^3=\binom{1\,2\,3}{3\,1\,\mathbf  2} & f_{\sigma}(3)=2 \\
	\sigma^2=\binom{1\,2}{2\,\mathbf  1} & f_{\sigma}(2)=1 \\
	\sigma^1=\binom{1}{\mathbf 1} & f_{\sigma}(1)=1.
\end{array}
\endgroup
\end{equation*}
Thus, $f_{\sigma}=112435487$.
\end{example}

\begin{proposition}[{See \cite[Prop. 3.5]{MantaciRakotondrajao2001}}]\label{prop:excAsIm}
For $f_{\sigma} \in \SEF_n$ we have that
$
 [n] \setminus \supp(f_{\sigma}) = \EXCv(\sigma).
$
In particular, $\exc(\sigma) = n - |\supp(f_{\sigma})|$.
\end{proposition}

Since subexcedant functions are seen as maps $g:[n] \to [n]$,
we have the notion of right-to-left minima, fixed points, etc., as defined in the previous section.
The following proposition is reminiscent of some results in \cite{BarilVajnovszki2017},
but they consider a different bijection between permutations and subexcedant functions.
\begin{proposition}\label{prop:rlmForSEF}
Let $\pi\in\symS_n$ and $f_{\pi}$ be the corresponding subexcedant function. 
Then
\begin{enumerate}[label=(\alph*)]
    \item $i \in \RLMi(\pi)\implies\pi(i)=f_{\pi}(i)$,
    \item $\RLMv(\pi) = \RLMv(f_{\pi})$,
    \item $\RLMi(\pi) = \RLMi(f_{\pi})$.
\end{enumerate}
\end{proposition}
\begin{proof}
We use induction over $n$, where the base case for $n=1$ is trivial.
Now let $\pi^n \in\symS_n$ and define $\pi^{n-1} \in \symS_{n-1}$  as
\begin{align}\label{eq:someLemmaRec}
    \pi^{n-1}(j) \coloneqq \begin{cases}
             \pi^{n}(n) &\text{if }\pi^n(j)=n\\
             \pi^{n}(j) &\text{otherwise},
			\end{cases}
\; \text{ so that } \; 
f_{\pi^n}(j)=\begin{cases}
             \pi^n(n) &\text{if }j=n\\
             f_{\pi^{n-1}}(j) &\text{otherwise.}
             \end{cases}
\end{align}
This is the same setup as in \cref{ex:permToSef}.
By induction hypothesis, 
$\pi^{n-1}$ fulfills properties \textit{(a), (b)}, and \textit{(c)}.

\medskip 
\noindent
Now suppose $i \in \RLMi(\pi^n)$. 
We must show that $\pi^n(i)=f_{\pi^n}(i)$.

\noindent
\textbf{Case $i=n$}:
 Here, $\pi^n(i) = f_{\pi^n}(i)$, as this follows immediately \eqref{eq:someLemmaRec}.

\noindent
\textbf{Case $i<n$}: 
Now, either $\pi^n(i)=n$ or $\pi^n(i)=\pi^{n-1}(i)$. 
But $\pi^n(i)=n$ is impossible since $\pi^n(i)$ is a right-to-left minima and $i<n$.
Hence 
\begin{equation}\label{eq:someLemma}
 \pi^n(i)=\pi^{n-1}(i) \text{ and } f_{\pi^n}(i)=f_{\pi^{n-1}}(i).
\end{equation}
\medskip 
Moreover, $\pi^n(i) < \pi^n(t)$ whenever $i < t \leq n$. But $\pi^n(t)=\pi^{n-1}(t)$ whenever $\pi^n(t)\neq n$. Thus, $\pi^{n-1}(i)= \pi^n(i)<\pi^n(t)=\pi^{n-1}(t)$ when $\pi^n(t)\neq n$.

\noindent If $\pi^n(t)=n$, then $\pi^{n-1}(t)=\pi^n(n)>\pi^n(i)=\pi^{n-1}(i)$, by the first formula in
 \eqref{eq:someLemmaRec}.  
 
\medskip
\noindent In any case, $\pi^{n-1}(i)<\pi^{n-1}(t)$, for $i<n$ and whenever $i < t \leq n-1$. So $i \in \RLMi(\pi^{n-1})$. This fact, together
with \eqref{eq:someLemma} and the induction hypothesis, 
finally gives
\[
i\in \RLMi(\pi^n) \implies \pi^n(i) = \pi^{n-1}(i)=f_{\pi^{n-1}}(i)=f_{\pi^n}(i),
\]
which completes the proof of property \textit{(a)}.

\medskip 
\noindent
We proceed with \textit{(b)}. 
By definition of $\pi^{n-1}$, we have that
\begin{align*}
    \RLMv(\pi^n) &=(\RLMv(\pi^{n-1})\cap [\pi^n(n)]) \; \cup \; \{\pi^n(n)\},  \\
    &=(\RLMv(f_{\pi^{n-1}})\cap[f_{\pi^n}(n)])\cup\{f_{\pi^n}(n)\}\\
    &=\RLMv(f_{\pi^n}),
\end{align*}
where the second equality follows from the induction hypothesis.

\medskip 
By the first property and the inductive hypothesis,
we have
\begin{align*}
    \RLMi(\pi^n) &= \{j \in \RLMi(\pi^{n-1}) : \pi^{n-1}(j) \leq \pi^n(n) \} \; \cup \; \{n\} \\
    &=\{j \in \RLMi(f_{\pi^{n-1}}): f_{\pi^{n-1}}(j) \leq f_{\pi^n}(n)\} \; \cup \; \{n\}\\
    &=\RLMi(f_{\pi^n}).
\end{align*}
This concludes the proof of property \textit{(c)}.
\end{proof}

We say that $f$ has a \defin{strict anti-excedance} at $i$ if $f(i)<i$.
Let $\defin{\aexc(f)}$ denote the number of strict anti-excedances in $f$.
\begin{proposition}[{See \cite[Prop. 4.1]{MantaciRakotondrajao2001}}]\label{prop:evenSEF}
The permutation $\sigma$ is even if and only if $\aexc(f_{\sigma})$
is even.
\end{proposition}

A \defin{fixed point} of $f \in \SEF_n$ is an integer $i \in [n]$ such that $f(i)=i$.
Moreover, $i$ is a \defin{multiple fixed point} of $f$ if:
\begin{enumerate}
	\item $f(i)=i$ and 
	\item there is some $j>i$ such that $f(j)=i$.
\end{enumerate}

\begin{proposition}[{See \cite[Prop. 3.8]{MantaciRakotondrajao2001}}]\label{prop:derangmentSEF}
We have that $\sigma \in \dArr_n$ if and only if all
fixed points of $f_{\sigma}$ are multiple.
\end{proposition}

\section{An involution and its consequences}\label{sec:involution}

A subexcedant function $f$ is \defin{matchless}
if it is of the form
\[
f= 1\,1\,2\,3\,4\dotsc k{-}1\,\,k\,\,k\dotsc k \quad \text{  for some $1\leq k \leq n{-}1$}.
\]
There are $n-1$ matchless subexcedant functions of length $n$.
For example, for $n=10$, the following subexcedant functions are matchless:
\begin{align*}
 &1111111111, \quad 1122222222, \quad 1123333333, \\
 &1123444444, \quad 1123455555, \quad 1123456666, \\
 &1123456777, \quad 1123456788, \quad 1123456789. 
\end{align*}

\begin{lemma}[Properties of matchless functions]\label{lem:propertiesOfMatchless}
Let $f_{\sigma} \in \SEF_n$ be matchless.
Then $\sigma =(1\,\,k{+}1\,\,k{+}2 \; \dotsc\; n\,\,k\,\,k{-}1 \; \dotsc \; 2)$. Moreover,
 \begin{align*}
 (-1)^{\inv(\sigma)} = (-1)^{n-1}, \,\,
\EXCv(\sigma) = [n]\setminus [k],\text{ and } 
\RLMv(\sigma) = [k].
\end{align*}
\end{lemma}
\begin{proof}
The form of $\sigma$ follows directly from \eqref{eq:sefToPerm}.
Since $\sigma$ has only one cycle, its sign is $(-1)^{n-1}$. From the definition of $f_{\sigma}$, we have that
\[
 \supp(f_{\sigma}) = [k] \implies \EXCv(\sigma) = [n]\setminus [k],
\]
by \cref{prop:excAsIm}. Similarly, the last property follows from 
\cref{prop:rlmForSEF}.
\end{proof}
Note that for each $k \in [n]$, there is a unique matcheless 
subexcedant function, with $k$ excedances.
We shall see that this property gives a combinatorial interpretation of the right-hand side of
R.~Mantaci and F.~Rakotondrajao's identity in \eqref{eq:mainEquation}.

\subsection{The involution}
Let $\defin{\DSEF_n} \coloneqq \{ f_{\sigma} : \sigma\in \dArr_n \}$
  and $\defin{\DSEF^*_n} \coloneqq \{ f_{\sigma} : \sigma\in \dArr_n 
 \text{ and } f_{\sigma} \text{ is not matchless}\}$.
In other words, $\DSEF_n$ is the set of subexcedant functions 
corresponding to derangements of $[n]$. Note that every $f \in \DSEF_n$ must have at least two $1$'s
in its row representation. We also call $\sigma$ a \defin{matchless derangement} if $f_{\sigma}\in\DSEF_n$ is matchless, and we use \defin{$\dArr^*_n$} to denote the set of non-matchless derangements.
%

Our goal is now to define an involution $\fretherD:\DSEF_n\longrightarrow \DSEF_n$, 
with the following properties:
\begin{enumerate}[label=(\roman*)]
\item 
	The image is preserved, $\supp(\fretherD(f_\sigma )) = \supp( f_\sigma )$.
\item 
	The set of right-to-left minima is preserved,
	$\RLMv( \fretherD(f_\sigma ) ) = \RLMv( f_\sigma )$.
\item 
	The fixed-elements of $\fretherD$ consist of the matchless subexcedant functions.
\item 
	The sign is reversed, $\sgn(\fretherD(f_\sigma )) = -\sgn( f_\sigma )$,
	whenever $f_\sigma \in \DSEF_n^*$.
\end{enumerate}

\bigskip 

We shall define $\defin{\fretherD}:\DSEF_n \longrightarrow \DSEF_n$ below,
where $\defin{f_{\tau}}$ is short for $\fretherD(f_{\sigma})$.
First, if $f_\sigma$ is matchless, we set $f_\tau\coloneqq f_\sigma$.
Now we fix some $f_{\sigma}\in \DSEF^*_n$ and 
let 
\begin{equation*}\label{eq:supp}
  \supp(f_{\sigma})=\{\defin{ \img_1, \img_2, \img_3, \dotsc, \img_{\ell} }\}.
\end{equation*}
 Note that $\img_1=1$ and  since $f_\sigma$ is non-matchless,
we know that $\ell \geq 2$ in $\supp(f_{\sigma})$.
With these preparations, we define two auxiliary maps, 
$\fixMap_i$, $\unfixMap_i$ on subexcedant functions.
For $i \in \{2,\dotsc,\ell\}$,
\begin{equation*}
   \defin{\fixMap_i}(f_\sigma)(\img_i) \coloneqq \img_i, \qquad 
   \defin{\unfixMap_i}(f_\sigma)(\img_{i}) \coloneqq\img_{i-1}
\end{equation*}
while the remaining entries of $f_\sigma$ are untouched.
For  $i \in \{2,\dotsc,\ell\}$, we say that \defin{$f_\sigma$ satisfies \eqCond[i]} 
(or simply \defin{\eqCond[i] holds} if $f_\sigma$ is clear from the context)
if the three conditions
\begin{equation}\label{eq:case34Conditions}
\begin{aligned}
  f_{\sigma}(\img_i) < \img_i < \img_\ell,\quad
  f^{-1}_\sigma(1) = \{1,2\}, \text{ and }
  \{\img_i + 1\}\subsetneq f^{-1}_\sigma(\img_i)
\end{aligned} \tag{\text{\eqCond[i]}},
\end{equation}
hold. 
Note that 
\[
  \{\img_i + 1\}\subsetneq f^{-1}_\sigma(\img_i) \iff 
  f_\sigma(\img_i + 1) = \img_i \text{ and } |f^{-1}_\sigma(\img_i)| \geq 2.
\]

\medskip 

Now let $i\in\{2,\dotsc,\ell\}$ be the \emph{smallest} element
satisfying one of the cases below,
and let $f_\tau$ be given as described in each case.
\begin{itemize}[leftmargin=2cm]
 \item[\textbf{Case \caseOne[i]:}] If $f_{\sigma}(\img_i) = \img_i$,
	then $f_\tau \coloneqq \unfixMap_i(f_\sigma)$.

 \item[\textbf{Case \caseTwo[i]:}] If $f_{\sigma}(\img_i) < \img_i$ and $|f^{-1}_\sigma(1)| \geq 3$,
 then $f_\tau \coloneqq \fixMap_i(f_\sigma)$.
 
 \item[\textbf{Case \caseThree[i]:}] \label{case3a} 
If \eqCond[i] holds and $f_\sigma(\img_{i+1}) = \img_{i+1}$, then
$
 f_\tau \coloneqq \unfixMap_{i+1}(f_\sigma).
$

\item[\textbf{Case \caseFour[i]:}] \label{case4} 
If \eqCond[i] holds and $f_\sigma(\img_{i+1}) < \img_{i+1}$, then
$
 f_\tau \coloneqq \fixMap_{i+1}(f_\sigma).
$
\end{itemize}
Note that for the same $i$, the four cases are mutually exclusive.
We emphasize that by saying that a case with subscript $i$ holds,
this particular $i \geq 2$ is the smallest $i$ for which the conditions one of the four cases hold.

\begin{example}\label{eg:main}
	Consider the following four subexcedant functions in $\DSEF_7$.
	\begin{enumerate}
		\item Let $f_\sigma=1133535$. Then 
		$\supp(f_\sigma)=\{1, 3, 5\}$ and 
		$2$ is the smallest index greater than 
		$1$ with
		$f_\sigma(\img_2)=f_\sigma(3)=3$. Hence, $f_\sigma$ is in case 
		\caseOne[2] and $f_\tau= 
		\unfixMap_2(f_\sigma)=1113535$.
		\item Now let $f_\sigma=1121355$. Then 
		$\supp(f_\sigma)=\{1, 2, 3, 5\}$. Since $f_\sigma(2)<2$ 
		and $|f_\sigma^{-1}(1)|=3$, then $f_\sigma$ is in case \caseTwo[2]. Thus, $f_\tau= 
		\fixMap_2(f_\sigma)=1221355$.
		\item Suppose that $f_\sigma=1123535$, then 
		$\supp(f_\sigma)=\{1, 2, 3, 5\}$. 
		The index $2$ does not satisfy any of the four cases. 
		So, we consider the next integer $i=3$. We note that \eqCond[3] holds and in 
		addition, 
		$f_\sigma(\img_{4})=f_\sigma(5)=5$. 
		Hence, $f_\sigma$ fulfills \caseThree[3]  and 
		$f_\tau= 
		\unfixMap_{i+1}(f_\sigma)=\unfixMap_{4}(f
		_\sigma)=1123335$.
		\item Now take $f_\sigma=1123445$. Then 
		$\supp(f_\sigma)=\{1, 2, 3, 4, 
		5\}$. None of the four cases for $f_\sigma$ are fulfilled 
		with $i\in\{2,3\}$. However, $f_\sigma$ 
		satisfies \eqCond[4] and 
		$f_\sigma(\img_{5})=f_\sigma(5)=4<\img_{5}$. 
		Thus, we are in \caseFour[4] and $f_\tau=\fixMap_{5}(f_\sigma)=1123545$.
	\end{enumerate}
\end{example}

\begin{remark}\label{rem:case2}
 Suppose \caseTwo[i] applies for $f_\sigma$.
 Then, for sure $f_\sigma(\img_2) < \img_2$, since otherwise, we would be in the case \caseOne[2]. 
 Hence, \caseTwo[i] may only apply when $i=2$.
\end{remark}

We have several things that need to be proved.
In \cref{lem:fretherDIsWellDefined} we show that $\fretherD$ is well-defined,
and in \cref{lem:correctRange}, we show that the range is correct.
In \cref{lem:imPreserving}, we show that $\fretherD$ preserves the image. Finally, in \cref{lem:rlmPreserving}, we show that $\fretherD$ preserves the right-to-left minima set.
In \cref{lem:signReversing} and \cref{lem:involution}, we show that  $\fretherD$ is sign-reversing on $\DSEF^*_n$ and $\fretherD$ is indeed an involution, respectively.
\smallskip

It is clear from the definition of $\fretherD$ that at most one of the cases applies for any $f_\sigma\in\DSEF^*_n$. For the well-definedness of $\fretherD$, we must also verify that at least one of the cases applies.

\begin{lemma}[Well-defined]\label{lem:fretherDIsWellDefined}
Let $f_\sigma \in \DSEF_n$ with $\ell$ elements in its image.
If none of the four cases ($\caseOne,\caseTwo,\caseThree,\caseFour$)
applies to $f_\sigma$, then $f_\sigma$ is matchless.

Moreover, if no $i \in \{2,\dotsc,t\}$ fulfills any of $\caseOne_i,\caseTwo_i,\caseThree_i,\caseFour_i$,
conditions for some $t\in[\ell]$,
and either $t=\ell$ or cases \caseOne[t+1] and \caseTwo[t+1] do not hold, then 
\begin{equation}\label{eq:sefPartialset}
f_\sigma(j)=\max\{1,j-1\},\text{ for all }j\in[t+1].
\end{equation}
Consequently, the prefix
\begin{equation}\label{eq:sefPartial}
	f_\sigma(1)\; f_\sigma(2)\; \dotsc \; f_\sigma(t)  \; f_\sigma(t+1) 
\end{equation}
is matchless. 
In addition, if $\ell=t$, then 
\begin{equation}\label{eq:matchlessfsigma}
	f_\sigma = 1\; 1\;2\;3\; \dotsc \; \ell{-}1 \; \ell \; \ell \; \dotsc \; \; \ell,
\end{equation}
which is matchless. Otherwise, 
\begin{equation}\label{eq:leftover}
\{f_\sigma(t+2),\dotsc,f_\sigma(n)\}=\{\img_{t+1},\dotsc,\img_\ell\}.
\end{equation}
\end{lemma}
\begin{proof}
We first note that \eqref{eq:sefPartial} follows immediately from \eqref{eq:sefPartialset} and  $\img_i=i$, for $i\in[t]$ by \eqref{eq:sefPartial}. The main statement follows from considering $t=\ell$ in \eqref{eq:sefPartial}.
We shall use induction on $t$ in order to prove \eqref{eq:sefPartialset}, \eqref{eq:matchlessfsigma}, and \eqref{eq:leftover}.

\medskip

\noindent\textbf{Base case $t=1$}: 
In this case, $\{2,\dotsc,t\}$ is empty.
If $\ell=t=1$, then $f_\sigma=1\,1\,1\,\dotsm\,1\,1$ (which is matchless). 
Otherwise, suppose that the cases \caseOne[t+1] and \caseTwo[t+1] do not hold. 

Since case \caseOne[2] is not fulfilled, then
$f_{\sigma}(\img_2)<\img_2$ so $f_{\sigma}(\img_2)=1$.
Hence, $f_{\sigma}(1)=1$ and 
$f_{\sigma}(2)=1$, otherwise $f_{\sigma}(2)=2$ 
which would violate our assumption.

Since case \caseTwo[2] is not fulfilled, although
$f_{\sigma}(\img_2)<\img_2$, then $|f_{\sigma}^{-1}(1)|<3$.
Thus, $f_{\sigma}^{-1}(1)=\{1,2\}$.
 Consequently, $\img_2=2$, since else $f_\sigma(3)=3$ and \caseOne[2] would be fulfilled. Hence, \eqref{eq:leftover} follows.

\bigskip 
\noindent
\textbf{Induction hypothesis:}
Suppose the statements hold for $t=k$, 
for some $k\geq 1$. We shall prove that they hold for $t=k+1$. 
\medskip

For this purpose suppose that none of the cases (\caseOne,\caseTwo,\caseThree,\caseFour)
holds for $i \in \{2,3,\dotsc,k+1\}$ and either $\ell=t=k+1$ or cases 
\caseOne[k+2] and \caseTwo[k+2] are not satisfied. Then, by the induction hypothesis, $f_\sigma(j)=\max\{1,j-1\},\text{ for }j\in[k+1]$ and $f_\sigma$ starts with $1\,1\,2\,3\,\dotsm\,k{-}1\,k$,
which is matchless. Since $\ell>k$ and the 
two cases (\caseOne[k+1], \caseTwo[k+1]) are not 
fulfilled, (by the induction 
hypothesis) none of the elements in $[k]$ 
belongs to 
$\{f_\sigma(k+2),\dotsc,f_\sigma(\ell)\}$. 
So, $f_\sigma(k+2)\in\{k+1, k+2\}$. We also have 
$\img_i=i$, for $i\in[k]$. We claim that 
$\img_{k+1}=k+1$. Otherwise, 
$\img_{k+1}>k+1$ and then 
$f_\sigma(\img_{k+1})=\img_{k+1}$, which would satisfy case \caseOne[k+1]. 
\medskip

If $\ell=k+1$, then $f_\sigma(k+2)=\img_{k+1}=k+1$ and
\[
f_\sigma=1\,1\,2\,3\,\dotsm\,k{-}1\,k\,k{+}1\,k{+}1\,\dotsm\,k{+}1,
\]
 indeed \eqref{eq:sefPartialset} and \eqref{eq:matchlessfsigma} holds.
\medskip

Else, $f_\sigma(k+2)=k+1$ and $f_\sigma$ starts with 
$1\,1\,2\,3\,\dotsm\,k{-}1\,k\,k{+}1$ since cases 
\caseOne[k+2] and \caseTwo[k+2] are not fulfilled. Thus, \eqref{eq:sefPartialset} holds. And since 
neither of the two cases (\caseThree[i], 
\caseFour[i]) holds for $i\in[k+1]$, at least 
one of the conditions in 
\eqCond[i] is not 
fulfilled. However, 
$f_\sigma(\img_i)<\img_i<\img_\ell$ (since $\ell>k+1$) and 
$f_\sigma(\img_i+1)=\img_i$, for all 
$i\in[k+1]$. Moreover, 
$f^{-1}_\sigma(1)=\{1,2\}$. Thus, 
$|f^{-1}_\sigma(\img_i)|=1$, for all 
$i\in\{2,\dotsc,k+1\}$. Hence, none of the elements in $[k+1]$ 
belongs to 
$\{f_\sigma(k+3),\dotsc,f_\sigma(\ell)\}$, which proves \eqref{eq:leftover}.
\end{proof}

\begin{remark}\label{rem:casefour}
If either \caseThree[i] or \caseFour[i] holds, then $\ell>i$, and both
\eqref{eq:sefPartialset} and 
\eqref{eq:leftover} hold for $t=i-1$. If, in 
particular, case \caseFour[i] is fulfilled, then 
$f_{\sigma}(\img_{i+1})=\img_i$ since 
$\img_{i+1}>f_{\sigma}(\img_{i+1})\in\{\img_{i},\dotsc,\img_{\ell}\}$.
\end{remark}

\begin{lemma}[Correct range]\label{lem:correctRange}
If $f_{\sigma}\in \DSEF_n$, then $f_\tau \coloneqq\fretherD(f_{\sigma})\in \DSEF_n$.
\end{lemma}
\begin{proof}
	If $f_\sigma$ is matchless, then we are done. 
	Suppose that $f_{\sigma}\in \DSEF^*_n$ and $i\geq 2$ satisfies one the cases in (\caseOne[i], \caseTwo[i], \caseThree[i], \caseFour[i]).
By \cref{prop:derangmentSEF}, it suffices to show that all fixed-points of $f_\tau$ are multiple.
\medskip

In the case of either \caseOne[i] or \caseThree[i], there will be no new fixed point created in $f_\tau$ since $f_\tau=\unfixMap_r(f_\sigma)$, for $r\in\{i,i+1\}$. So all the fixed points of $f_\sigma$ remain multiple in $f_\tau$ too except for $\img_r$, which is not fixed in $f_\tau$.
\medskip

If the case \caseTwo[i] fulfilled, then $i=2$ by \cref{rem:case2}, and $f_\sigma(\img_2)=1$. Moreover,
$f_{\tau}(\img_2)=\img_{2}$ and there is some $j>\img_2$ such 
that $f_{\tau}(j)=f_{\sigma}(j)=\img_{2}$. That is, 
$\img_2$ is a multiple fixed point in $f_{\tau}$. And so is $1$ since $|f^{-1}_\sigma(1)|\geq 3$ implies $|f^{-1}_\tau(1)|\geq 2$.
\medskip

If the case \caseFour[i] fulfilled, then $f_{\sigma}(\img_{i+1})=\img_i$ by \cref{rem:casefour},
and there is some $j>\img_{i+1}$ such that $f_{\tau}(j)=f_{\sigma}(j)=\img_{i+1}$.
Consequently, $f_{\tau}(\img_{i+1})=\img_{i+1}$ is a multiple fixed point in $f_{\tau}$ while $\img_i$ is not a fixed point in both $f_\sigma$ and $f_\tau$.
\end{proof}

\begin{lemma}[Image-set preserving]\label{lem:imPreserving}
For $f_{\tau}=\fretherD(f_{\sigma})$, we have 
\begin{equation}\label{eq:imPreserving}
\supp(f_{\sigma}) = \supp(f_{\tau}) \text{ and }
\EXCv(\sigma) = \EXCv(\tau). 
\end{equation}
\end{lemma}
\begin{proof}
	First note that $\supp(f_{\tau})\subseteq \supp(f_{\sigma})$, which clearly follows from 
	the definition of $\fretherD$. Now suppose that one of the cases in (\caseOne[i], \caseTwo[i], \caseThree[i], \caseFour[i]) 
	is satisfied for $i\geq 2$. Recall that the map $\fretherD$ first removes an element in position $\img_r$, 
	for $r\in\{i, i+1\}$, in $f_\sigma$ and then insert another element on the same position to 
	obtain $f_\tau$. So, it is enough to show that the 
	removed element is in $\supp(f_\tau)$ for $\supp(f_{\sigma}) = \supp(f_{\tau})$ to hold.
\medskip

In the case of \caseOne[i] or \caseThree[i], $f_\tau=\unfixMap_r(f_\sigma)$ and there is some $j>\img_r$ such 
that $\img_r=f_{\sigma}(j)=f_{\tau}(j)$ since $\img_r$ is a multiple fixed point in $f_\sigma$ in these cases. So $\img_r\in \supp(f_{\tau})$.
\medskip

If \caseTwo[i] holds, then
$f_{\sigma}(\img_i)=1$, since $i=2$ (by \cref{rem:case2}). Moreover, $f_\tau(\img_i)=\img_i$. However, $|f^{-1}_\sigma(1)|\geq 3$. So, $|f^{-1}_\tau(1)|\geq 2$ and then $1\in \supp(f_{\tau})$. 
\medskip

Finally, suppose case \caseFour[i] is fulfilled. Then by \cref{rem:casefour}, $f_{\sigma}(\img_{i+1})=\img_i$. We can now conclude that $\img_i\in \supp(f_{\tau})$, since $|f^{-1}_\sigma(\img_i)|\geq 2$. 
\medskip

Therefore, the first equality in \eqref{eq:imPreserving} is proved, while the second follows from \cref{prop:excAsIm}. 
\end{proof}

\begin{lemma}\label{lem:rlmPreserving}
For $f_{\tau} = \fretherD(f_{\sigma})$ we have 
\begin{equation}\label{eq:ritoleft}
 \RLMv(f_{\sigma}) = \RLMv(f_{\tau}) \text{ and }
 \RLMv(\sigma) = \RLMv(\tau).
\end{equation}
\end{lemma}
\begin{proof}
Let $f_\sigma\in \DSEF_n$. 
If $f_\sigma$ is matchless, then $f_\tau=f_\sigma$ and  $\RLMv(f_\sigma)=\RLMv(f_\tau)$.
Suppose $f_\sigma$ is non-matchless, so that one of the four cases applies.

\begin{enumerate}[leftmargin=2cm]
 
	\item[\textbf{Case \caseOne[i]:}] 
	Then $f_\sigma(\img_i)=\img_i$ and 
	$f_\tau(\img_i)=\img_{i-1}$. 
	Moreover, there is some $j>\img_i$ such that $f_\tau(j)=f_\sigma(j)=\img_i$. 
	The property of $\img_i$ being a right-to-left minimum in $f_\sigma$
	as well as $f_\tau$ is determined either at the position $j$
	or to the right of $j$. Hence, replacing $\img_{i}$ by $\img_{i-1}$ at position
	$\img_i$, preserves $\img_i$
	being (or not) a right-to-left minimum.
	
	\begin{itemize}
		\item If $i\geq 3$, then $\img_i$ is the leftmost occurrence of $\img_{i-1}$ in $f_\tau$,
		since $i$ is the smallest such that $f_\sigma(\img_{i})=\img_i$.
		Since $\supp(f_\sigma)=\supp(f_\tau)$, there is some $k>\img_i$
		such that $f_\tau(k)=f_\sigma(k)=\img_{i-1}$.
		So, $\RLMv(f_\sigma)=\RLMv(f_\tau)$.
		
		\item If $i=2$, then $\img_{i-1}=\img_1=1\in\RLMv(f_\tau)$. Since $1\in \RLMv(f_\sigma)$, then $\RLMv(f_\sigma)=\RLMv(f_\tau)$.
	\end{itemize}
	
	\item[\textbf{Case \caseTwo[i]:}]
	Then $i=2$ and $f_\sigma(\img_2)=\img_1=1$
	and $f_\tau(\img_2)=\img_2$. 
	Moreover, there is some $k>\img_2$
	such that $f_\tau(k)=f_\sigma(k)=\img_2$.
	Since the right-to-left minimum property of $\img_2$
	is determined at or to the right of the $k^\thsup$
	position and $1\in \RLMv(f_\tau)$, we have $\RLMv(f_\sigma)=\RLMv(f_\tau)$.

	\item[\textbf{Case \caseThree[i]:}]
	Then $f_\sigma(\img_{i+1})=\img_{i+1}$
	and $f_\tau(\img_{i+1})=\img_{i}$.
	We claim that $\img_i\in\RLMv(f_\sigma)$.
	Otherwise, there are $r<i$ and $s>j$,
	such that $f_\sigma(s)=\img_r$, where $j$ is the rightmost position of $\img_i$ in $f_\sigma$.
	But now, $|f_\sigma^{-1}(\img_r)| \geq 2$ (by \cref{lem:fretherDIsWellDefined})
	and so \eqCond[r] is fulfilled and case \caseFour[r] holds.
	This contradicts the choice of $i$ being minimal, and the claim follows.
	
	Since $\img_{i+1}$ being a right-to-left minimum is determined at some other
	position $k>\img_{i+1}$ where $f_\tau(k)=f_\sigma(k)=\img_{i+1}$,
	we can conclude that $\RLMv(f_\sigma)=\RLMv(f_\tau)$.
	
	\item[\textbf{Case \caseFour[i]:}]
	There is some $i\geq 2$ such that $f_\sigma(\img_{i+1})=\img_i<\img_{i+1}$ (by \cref{rem:casefour})
	and $f_\tau(\img_{i+1})=\img_{i+1}$.
	We also know that $\img_i\in\RLMv(f_\sigma)$. 
	Since $\img_{i+1}>\img_i$, we have $\img_i\in\RLMv(f_\tau)$.
	
	Now, $\img_{i+1}$ being a right-to-left minimum
	is determined at some position $k>\img_{i+1}$ where 
	$f_\tau(k)=f_\sigma(k)=\img_{i+1}$.
	Hence, we can conclude that $\RLMv(f_\sigma)=\RLMv(f_\tau)$.
	\end{enumerate}
The second equality in \eqref{eq:ritoleft} follows  from \cref{prop:rlmForSEF}.
\end{proof}

\begin{lemma}\label{lem:signReversing}
If $f_{\sigma}\in \DSEF^*_n$, then 
\[
 \aexc(\fretherD(f_{\sigma})) \in \left\{ \aexc(f_{\sigma})-1,\; \aexc(f_{\sigma})+1 \right\}.
\]
Moreover, if $f_\tau = \fretherD(f_{\sigma})$,
then $\sigma$ and $\tau$ have different parity.
\end{lemma}
\begin{proof}
The first statement follows from the fact that $\fixMap_i$
and $\unfixMap_i$ decreases and increases, respectively, 
the number of strict anti-excedances by one.
The second statement follows from the first by \cref{prop:evenSEF}.
\end{proof}

\begin{lemma}\label{lem:involution}
	The map $\fretherD:\DSEF_n\longrightarrow \DSEF_n$
	is an involution.
\end{lemma}
\begin{proof}
	Let $f_\sigma\in \SEF_n$ with $ \supp(f_{\sigma})=\{ \img_1, \img_2, \img_3, \dotsc, \img_{\ell} \}$ and set
	$f_\tau=\fretherD(f_\sigma)$. For matchless $f_\sigma$, there is nothing to show.
Now assume that one of (\caseOne[i], \caseTwo[i], \caseThree[i], \caseFour[i]) holds for $f_\sigma$ 
and one of (\caseOne[i'], \caseTwo[i'], \caseThree[i'], \caseFour[i']) holds $f_\tau$, for $i,i'\in\{2, \dotsc, l\}$.
	
\begin{enumerate}[leftmargin=1.6cm]
 	\item[\textbf{Case \caseOne[i]:}] 
 	  We have $f_\sigma(\img_i)=\img_i$, 
 	  and $f_\tau=\unfixMap_i(f_\sigma)$.
		
		\begin{itemize}[leftmargin=-1cm]
			\item
			If $\vert f_\tau^{-1}(1)\vert \geq 3$, then $i=2$.
			
			\medskip  
			Suppose $i\geq 3$. Then $|f_\sigma^{-1}(1)|<3$ since otherwise $f_\sigma$ would be in case \caseTwo[j] for some $j < i$.
			Thus,  $|f_\sigma^{-1}(1)| < |f_\tau^{-1}(1)|$ and then there is some $r\in [n]$ such that $f_\tau(r) = 1\neq f_\sigma(r)$. However, $r=\img_i$ since $f_\tau$ and $f_\sigma$ differs
			only on position $\img_i$. So, $f_\tau(\img_i)=1=\img_{i-1}$  and this only happens if $i=2$.
			
			\medskip 
			Hence, $i=2$ and  $\img_2\in f_\tau^{-1}(1)\setminus f_\sigma^{-1}(1)$. There is now some 
			$h>\img_2$ such that 
			$f_\tau(h)=f_\sigma(h)=
			\img_2$ since 
			$f_\sigma\in\DSEF_n$. 
			So, $f_\tau$ satisfies 
			the conditions for 
			case \caseTwo[i'] with $i'=i$. 
			It follows that $\fretherD(f_\tau)=f_\sigma$.

			\medskip
			
			\item If $\vert f_\tau^{-1}(1)\vert = 2$, then $f_\tau^{-1}(1)=f_\sigma^{-1}(1)$. Because  
			we always have $f_\sigma^{-1}(1)\subseteq f_\tau^{-1}(1)$ in 
			\caseOne[i] and \caseThree[i],
			and since $|f_\sigma^{-1}(1)| \geq 2$, we must have equality.
			Moreover, $i \geq 3$ since otherwise $f_\tau(\img_2)=1$ and then $|f_\tau^{-1}(1)|>|f_\sigma^{-1}(1)|$.
			
			\medskip
			\noindent
			By applying \cref{lem:fretherDIsWellDefined} for $t=i-2$,  the first $i-1$ entries of $f_\sigma$ and $f_\tau$ are
			\[
			1, 1, 2, 3,\dotsc , i{-}3,\, i{-}2,
			\]
			and $\img_j=j$, for all $j\in[i-2]$. In addition, $|f^{-1}_\sigma(j)|=1$ otherwise $f_\sigma$ would lie in case \caseFour[j] holds for some $j\in[i-2]$. Hence, $f_\sigma(i) \in \{i{-}1,i\}$.

			\medskip 
			\noindent 
			If $f_\sigma(i)=i$, then $\img_i=i$ (since $f_\sigma$ is in case \caseOne[i]) and  $f_\tau(i)=\img_{i-1}=i-1$ since $i-2=\img_{i-2}$. 
			Thus, there exists some $s>i$ such that  $f_\tau(s)=f_\sigma(s)=i-1$. Then
			\begin{align*}
				f_\sigma&=1\,1\,2\,3\,\dotsm\,i{-}3\,\,i{-}2\,\,i\,\dotsm\,i{-}1\,\dotsm\,i\,\dotsm\\
				f_\tau&=1\,1\,2\,3\,\dotsm\,i{-}3\,\,i{-}2\,\,i{-}1\,\dotsm\,i{-}1\,\dotsm\,i\,\dotsm\\
				&\text{ or }\\
				f_\sigma&=1\,1\,2\,3\,\dotsm\,i{-}3\,\,i{-}2\,\,i\,\dotsm\,i\,\dotsm\,i{-}1\,\dotsm\\
				f_\tau&=1\,1\,2\,3\,\dotsm\,i{-}3\,\,i{-}2\,\,i{-}1\,\dotsm\,i\,\dotsm\,i{-}1\,\dotsm,
			\end{align*}
		Now we can see that $f_\tau$ satisfies the conditions in \eqCond[i'], for $i'=i-1$:
		\[
		f_\tau(\img_{i-1})=f_\tau(i-1)=i-2<\img_{i-1}<\img_{\ell}\text{ (since $\ell\geq i$)},\,\,
		f^{-1}_\tau(1)=\{1,2\}, \]
		\[f_\tau(\img_{i-1}+1)=
        f_\tau(i)=i-1=\img_{i-1}, \text{ and }
		|f^{-1}_\tau(\img_{i-1})|\geq 2.
		\]
		Hence, $f_\tau$ fulfills case \caseFour[i'] for $i'=i-1$ and $\fretherD(f_\tau)=f_\sigma$.
		
		\medskip
		If $f_\sigma(i)=i-1$, then $\img_{i-1}=i-1$ and $\img_{i}>i$. Thus, $f_\tau(\img_i)=i-1$. Moreover, $|f^{-1}_\sigma(i-1)|=1$. Otherwise, 
		$f_\sigma$ would satisfy \eqCond[i] whence either \caseThree[i] or \caseFour[i] would be fulfilled. This implies that $f^{-1}_\tau(i-1)=\{i, \img_i\}$. Now, it is easy to see that $f_\tau$ satisfies the conditions in \eqCond[i'] for $i'=i-1$. Therefore, $f_\tau$ fulfills the case \caseFour[i'] for $i'=i-1$ and $\fretherD(f_\tau)=f_\sigma$.
		\end{itemize}
		
		\item[\textbf{Case \caseTwo[i]:}] 
		Then $i=2$ and $f_\tau = \fixMap_2(f_\sigma)$.
		We have that $f_\tau(\img_2)=\img_2$ and $f_\tau$ belongs to \caseOne[i'] for $i'=i$,
		so $\fretherD(f_\tau)=f_\sigma$.

	\item[\textbf{Case \caseThree[i]:}] 
	In this case, \eqCond[i] hold, which state that 
	\[
	f_{\sigma}(\img_i) < \img_i < \img_\ell,\,\,
	f^{-1}_\sigma(1) = \{1,2\},\,\, f_\sigma(\img_i+1)=i, \text{ and }
  |f^{-1}_\sigma(\img_i)|\geq 2.
	\]
	And also $f_\sigma(\img_{i+1})=\img_{i+1}$.
	
	\medskip
	\noindent
	Since $f_\tau = \unfixMap_{i+1}(f_\sigma)$ and $i\geq 2$, we have $f_\tau(\img_i)<\img_i<\img_\ell$
	and $f_\tau^{-1}(1)=\{1,2\}$. And also since
	$f_\sigma(\img_i+1)=\img_i$ and  
	$f_\sigma(\img_{i+1})=\img_{i+1}$, we have $\img_i+1\neq \img_{i+1}$. 
	This implies that $f_\tau(\img_i+1)=f_\sigma(\img_i+1)=\img_i$. We also have that $|f_\tau^{-1}(\img_i)| \geq 3$,
	since $|f_\sigma^{-1}(\img_i)|\geq 2$ and $f_\tau(\img_{i+1})=\img_i$. 
	Hence, $f_\tau$ satisfies the conditions in \eqCond[i'] for $i'=i$ and then it belongs to \caseFour[i'].
	It follows that $\fretherD(f_\tau)=f_\sigma$.	
		\item[\textbf{Case \caseFour[i]:}] 
		Again, we have \eqCond[i] for $f_\sigma$, and $f_\sigma(\img_{i+1})<\img_{i+1}$. 
		Moreover,  $f_\sigma(\img_{i+1})=\img_{i}$ (from \cref{rem:casefour}).
		Recall, $f_\tau = \fixMap_{i+1}(f_\sigma)$. We also have that $f_\tau(\img_i)=f_\sigma(\img_i)<\img_i<\img_l$, $f_\tau^{-1}(1)=\{1,2\}$, and $f_\tau(\img_{i+1})=\img_{i+1}$. 
		We shall now consider two subcases.
		\begin{itemize}
			\item \emph{Suppose $\img_i+1< \img_{i+1}$ and $\vert f_\sigma^{-1}(\img_i)\vert\geq 3$.} We have that 
			$f_\tau(\img_i+1)=f_\sigma(\img_i+1)=\img_i$. Then $|f_\tau^{-1}(\img_i)|\geq 2$ (then \eqCond[i] is satisfied for $f_\tau$) and $f_\tau$ belongs to case \caseThree[i'] for $i'=i$.
			\item \emph{Otherwise,}  $|f_\tau^{-1}(\img_i)|=1$ if $\img_i+1< \img_{i+1}$ and $\vert f_\sigma^{-1}(\img_i)\vert=2$. On the other hand, $f_\tau(\img_i+1)=f_\tau(\img_{i+1})=\img_{i+1}$ if $\img_i+1=\img_{i+1}$.  Then \eqCond[i] will not be satisfied for $f_\tau$ in both cases. Therefore, $f_\tau$ lies in case \caseOne[i'] for $i'=i+1$.
		\end{itemize}
	In both cases, $\fretherD(f_\tau)=f_\sigma$.
	\end{enumerate}
\end{proof}

\begin{remark}
In \cref{tab:casesCombinations}, we give an overview under what circumstances 
a subexcedant function belonging to a case, is mapped to a different case.
\begin{table}[!ht]
\begin{center}
\begin{tabular}{ l p{2cm} p{1.7cm} p{1.7cm} p{2.5cm} }
\; & $f_\sigma\in\caseOne[i]$ & $f_\sigma\in\caseTwo[i]$ & $f_\sigma\in\caseThree[i]$ & $f_\sigma\in\caseFour[i]$ \\
\toprule
\rowcolor{gray!10}
$f_\tau\in\caseOne[i']$   & $\emptyset$  & Always, $(i'=i)$ & $\emptyset$ & $f_\tau$ not fulfill \eqref{eq:case34Conditions}, $(i'=i)$ \\
$f_\tau\in\caseTwo[i']$   & $|f_\tau^{-1}(1)|\geq 3$, $(i'=i)$  & $\emptyset$ & $\emptyset$ & $\emptyset$\\
\rowcolor{gray!10}
$f_\tau\in\caseThree[i']$ & $\emptyset$  & $\emptyset$ & $\emptyset$ & $f_\tau$ fulfills \eqref{eq:case34Conditions},  $(i'=i+1)$ \\
$f_\tau\in\caseFour[i']$  & $|f_\tau^{-1}(1)| =2$, $(i'=i-1)$  & $\emptyset$ & Always, $(i'=i)$ & $\emptyset$\\
\bottomrule
\end{tabular}
\medskip
\caption{
When $f_\tau = \fretherD(f_\sigma)$, we have the combinations under the 
conditions described in the cells of the table.
}\label{tab:casesCombinations}
\end{center}
\end{table}
\end{remark}

For example, we can see that the subexcedant function
\begin{enumerate}
	\item $f_\sigma=1133535$ satisfies case \caseOne[2]. However, its image, $f_\tau=1113535$, lies in case \caseTwo[2].
	\item $f_\sigma=1124545$ satisfies case \caseOne[3]. Then, its image, $f_\tau=1122545$, lies in case \caseFour[2].
	\item $f_\sigma=1121355$, which is in case \caseTwo[2], mapped to $f_\tau=1221355$ that belongs to case \caseOne[2]. 
	\item $f_\sigma=1123535$ satisfies case \caseThree[3]. Nevertheless, the image $f_\tau=1123335$ appears in case \caseFour[3].
	\item  $f_\sigma=11233353$ is in case \caseFour[3]. The image, $f_\tau=11235353$, is in case \caseThree[3].
	\item $f_\sigma=1123445$ is in case \caseFour[4]. However, its image, $f_\tau=1123545$, is in case \caseOne[5].
\end{enumerate}

\medskip

We conclude this subsection by listing all properties proved for $\fretherD$.
\begin{corollary}\label{cor:fretherDProps}
The map $\fretherD:\DSEF_n\longrightarrow \DSEF_n$
is an involution with the following properties.
\begin{enumerate}[label=(\roman*)]
\item 
	The image is preserved, $\supp(\fretherD(f_\sigma )) = \supp( f_\sigma )$.
\item 
	The set of right-to-left minima is preserved,
	$\RLMv( \fretherD(f_\sigma ) ) = \RLMv( f_\sigma )$.
\item 
	Whenever $f_\sigma \in \DSEF^*_n$, 
	\[
	\aexc(\fretherD( f_\sigma ) ) = \aexc( f_\sigma)\pm 1.
	\]
\end{enumerate} 
\end{corollary}
We now have an involution on derangements $\defin{\frether}:\dArr_n \to \dArr_n$
by setting
\[
\defin{\frether(\sigma)} \coloneqq (\sefToPerm \circ \fretherD \circ \sefToPerm^{-1})(\sigma),
\text{ for $\sigma \in \dArr_n$}.
\]

\begin{corollary}\label{cor:frethermap}
The involution $\frether$ satisfies the properties below:
\begin{enumerate}[label=(\roman*)]
	\item The excedance value set is preserved,  $\EXCv(\frether(\sigma)) = \EXCv(\sigma).$
	\item 
	The set of right-to-left minima is preserved,
	$\RLMv( \frether(\sigma)) = \RLMv(\sigma )$.
	\item 
		Whenever $\sigma \in \dArr^*_n$, 
		$\sgn(\frether(\sigma) ) = -\sgn( \sigma)$.
\end{enumerate}
\end{corollary}
 
\subsection{Consequences}

Before stating the main theorem, we shall first 
introduce two auxiliary involutions on $ \symS_n$,
and prove some of their properties.
Let $\defin{\flipMap} : \symS_n \to \symS_n$ be the map 
\[
\defin{\flipMap(\sigma)(k)} \coloneqq n+1-\sigma(k) \quad \text{ for } \quad k\in [n],
\]
and let $\defin{\zetaMap} : \symS_n \to \symS_n$ be the composition
$
 \defin{\zetaMap} \coloneqq \flipMap^{-1} \circ (\;\cdot\;)^{-1} \circ \flipMap.
$
In other words,
\[
  \defin{\zetaMap(\sigma)(k)} \coloneqq n+1-\sigma^{-1}(n+1-k) \quad \text{ for } \quad k\in [n].
\] 
\begin{lemma}\label{lem:conjugateProps}
The map $\zetaMap$ is an involution, and
\begin{align*}
 \EXCv(\pi) &= \{ n+1-k : k \in \EXCi( \zetaMap(\pi)) \}, \\
 \RLMv(\pi) &= \{ n+1-k : k \in \RLMi( \zetaMap(\pi)) \},\\
 \FIX(\pi)  &= \{ n+1-k : k \in \FIX( \zetaMap(\pi)) \},\\
 \inv(\pi)  &= \inv(\zetaMap(\pi)).
\end{align*}
In particular, $\zetaMap$ restricts to a sign-preserving
involution $\zetaMap:\dArr_n \to  \dArr_n$.
\end{lemma}
\begin{proof}
It follows immediately from the definition that $\zetaMap$ is an involution. 
For the first property, let $i\in [n]$ and set $j \coloneqq n+1-i$.
We then see that
\begin{align*}
 i\in\EXCv(\pi) &\iff \pi^{-1}(i)<i \\
                &\iff n+1-\pi^{-1}(i)>n+1-i \\
                &\iff n+1-\pi^{-1}(n+1-j)>j \\
                &\iff \zetaMap(\pi)(j)>j \\
                &\iff i\in \{n+1-k:k \in \EXCi(\zetaMap(\pi)) \}.
\end{align*}
Now for the right-to-left minima, again with $j \coloneqq n+1-i$, we have 
\begin{align*}
i\in \{n+1-k:k \in \RLMi(\zetaMap(\pi)) \} &\iff 
\zetaMap(\pi)(j)<\zetaMap(\pi)(t) \text{ whenever } j<t \leq n \\
&\iff \pi^{-1}(n+1-j)>\pi^{-1}(n+1-t) \text{ whenever } j<t \leq n \\
&\iff \pi^{-1}(i)>\pi^{-1}(k) \text{ whenever } k\in[i-1] \\
&\iff \text{every $k\in[i-1]$ lies to the left of $i$ in $\pi$}\\
&\iff i\in\RLMv(\pi).
\end{align*}
Similarly, with $i \in [n]$, $j \coloneqq n+1-i$,
\begin{align*}
i\in\FIX(\pi) &\iff \pi^{-1}(i)=i \\
&\iff \pi^{-1}(n+1-j)=n+1-j\\
&\iff n+1-\pi^{-1}(n+1-j)=j\\
&\iff \zetaMap(\pi)(j)=j\\
&\iff j\in\FIX(\zetaMap(\pi)).
\end{align*} 
This last property also shows that $\zetaMap$ is an involution on $\dArr_n$.

Finally, we have that
\begin{align*}
	\inv(\zetaMap(\pi)) &= |\{ (i,j) : 1\leq i < j \leq n \text{ such that } \zetaMap(\pi)(i)>\zetaMap(\pi)(j) \}|\\
	&=|\{ (i,j) : 1\leq i < j \leq n \text{ such that }\pi^{-1}(n+1-i)<\pi^{-1}(n+1-j) \}|\\
	&=|\{ (n+1-k,n+1-l) : 1\leq l < k \leq n \text{ such that }\pi^{-1}(k)<\pi^{-1}(l) \}|\\
	&=|\{ (l',k') : 1\leq l' < k' \leq n \text{ such that }\pi^{-1}(k')<\pi^{-1}(l') \}|\\
	&=\inv(\pi),
\end{align*}
so $\zetaMap$ preserves the number of inversions.
In particular, $\zetaMap$ preserves the sign.
\end{proof}

We are now ready to prove the main theorem in this paper.
\begin{theorem}\label{thm:theMainTheorem}
 We have that 
 \begin{equation}\label{eq:mainResultInText}
 \sum_{\pi \in \dArr_n} (-1)^{\inv(\pi)} 
 \xvec_{\RLMv(\pi)} \yvec_{ \EXCv(\pi)} 
 =
 (-1)^{n-1} \sum_{j=1}^{n-1} x_1 \dotsm x_j \cdot y_{j+1}\dotsm y_{n}.
 \end{equation} 
Moreover, 
 \begin{equation}\label{eq:mainResultInText2}
 \sum_{\pi \in \dArr_n} (-1)^{\inv(\pi)} 
 \xvec_{\RLMi(\pi)} 
 \yvec_{\EXCi(\pi)}
 =
 (-1)^{n-1} \sum_{j=1}^{n-1} y_1 \dotsm y_j \cdot x_{j+1}\dotsm x_{n}.
 \end{equation} 
\end{theorem}
\begin{proof}
By applying the involution $\frether$ and 
using all the properties listed in \cref{cor:frethermap},
all terms in the left-hand side of \eqref{eq:mainResultInText} cancel, except the terms with $\pi \in \dArr^*_n$. Thus, the left-hand side of \eqref{eq:mainResultInText} equals
\begin{equation*}
\sum_{\pi\in  \dArr^*_n }
	(-1)^{\inv(\pi)} \xvec_{\RLMv(\pi)} \yvec_{\EXCv(\pi)}. 
\end{equation*}	
By \cref{lem:propertiesOfMatchless}, this sum is equal to
\begin{equation*}
\sum_{k=1}^{n-1}
(-1)^{n-1} \xvec_{[k]} \yvec_{[n] \setminus [k]},
\end{equation*}
which is the right-hand side of \eqref{eq:mainResultInText}.
\medskip

Let $\defin{\rho(S)}\coloneqq \{n+1-s : s\in S\}$ whenever $S\subseteq [n]$.
By applying the change of variables
$x_i \mapsto x_{n+1-i}$, $y_j \mapsto y_{n+1-j}$ on both sides of 
\eqref{eq:mainResultInText}, we get
\begin{align*}
 \sum_{\pi \in \dArr_n} (-1)^{\inv(\pi)} 
 \xvec_{\rho(\RLMv(\pi))}
  \yvec_{\rho(\EXCv(\pi))}
&= (-1)^{n-1} \sum_{j=1}^{n-1} x_n \dotsm x_{n+1-j} \cdot y_{n-j}\dotsm y_{1} \\
&= (-1)^{n-1} \sum_{j'=1}^{n-1} x_{j'+1} \dotsm x_{n} \cdot y_{1}\dotsm y_{j'}.
\end{align*}
Now by \cref{lem:conjugateProps},
\begin{align*}
\sum_{\pi \in \dArr_n} (-1)^{\inv(\pi)} 
\xvec_{\rho(\RLMv(\pi))}
\yvec_{\rho(\EXCv(\pi))}
\!
=
\!
\sum_{\pi \in \dArr_n} (-1)^{\inv(\zetaMap(\pi))} 
\xvec_{\RLMi(\zetaMap(\pi))}
\yvec_{\EXCi(\zetaMap(\pi))}.
\end{align*}
Since $\zetaMap$ sends $\dArr_n$ to $\dArr_n$, 
the last sum must be exactly the left-hand side of \eqref{eq:mainResultInText2}, and we are done.
\end{proof}

\begin{corollary}
By letting $x_j \to 1$ and $y_j \to t$, we have that 
\begin{equation*}
 \sum_{\pi \in \dArr_n} (-1)^{\inv(\pi)} t^{\exc(\pi)}
 =
 (-1)^{n-1}(t+t^2+\dotsb+t^{n-1}).
 \end{equation*}
By comparing coefficients of $t^k$,
we get \cref{eq:mainEquation}.
In a similar manner,
\begin{equation*}
\sum_{\pi \in \dArr_n} (-1)^{\inv(\pi)} t^{\rlm(\pi)}
 =
 (-1)^{n-1}(t+t^2+\dotsb+t^{n-1}).
\end{equation*}
\end{corollary}

\section{A proof using generating functions}\label{sec:genFunc}

We shall first define an involution
$\iota:\symS_n \to \symS_n$ such that for $\pi\in\symS_n$,
\begin{enumerate}
 \item $\EXCi(\iota(\pi))=\EXCi(\pi)$,
 \item $\sgn(\iota(\pi))=-\sgn(\pi)$ if $\iota(\pi)\neq\pi$,
 \item $\sgn(\pi)=(-1)^{\exc(\pi)}$ if $\iota(\pi)=\pi$,
 \item for each $E \subseteq [n-1]$, there is a unique $\pi$ with $\iota(\pi)=\pi$ such that $\EXCi(\pi)=E$.
\end{enumerate}
We shall now describe $\iota$,
which is essentially the one given in \cite{Mantaci1993}.

\begin{definition}
Define a mapping $\defin{\iota}: \symS_n \rightarrow \symS_n$ by 
$\iota(\pi)=\pi'$,
where $\pi'$ is obtained from $\pi$ by swapping $\pi(l)$ and $\pi(m)$,
where
\[
(l,m)=\max\{(i,j): 2\leq i<j\leq n\text{ and either }i, j\in \EXCi(\pi)\text{ or }i,j\notin \EXCi(\pi)\}
\] 
with respect to lexicographical order, so that $\EXCi(\pi')=\EXCi(\pi)$. 
It can be defined as $\iota(\pi)=(l,m)\pi $ if $\pi$ is in cycle form.
If there is no such $(l,m)$, then $\iota(\pi)=\pi$ and we say that $\pi$ is \defin{critical}.
\end{definition}
From the definition, it is clear that (1) and (2) hold.
We must show that (3) and (4) hold as well, which are done in \cref{prop:mantaciInvolution}.

\begin{lemma}
	Suppose $\pi \in \symS_n$, and that $i,j\in\EXCi(\pi)$ with $i<j$. If $\pi(i)>j$, then $\pi$ is not critical.
	
	Similarly, if $i',j'\notin\EXCi(\pi)$ such that $i'<j'$
	with $\pi(j')\leq i'$, then $\pi$ is not critical.
\end{lemma}
\begin{proof}
	After swapping $\pi(i)$ and $\pi(j)$,
	 both $i$ and $j$ remain excedances since $\pi(j)>j>i$
	and $\pi(i)>j$. Hence, $\pi$ is not critical as there is at least
	one pair of entries where  we can perform a swap as in the definition of $\iota$.
	A similar argument proves the second statement.
\end{proof}

\begin{corollary}\label{cor:critical}
	Suppose $\pi\in\symS_n$ with $\EXCi(\pi)=\{j_1,\dotsc,j_k\}$ and $[n]\setminus\EXCi(\pi)=\{i_1,\dotsc,i_{n-k}\}$. Then, $\pi$ is critical iff 
	\begin{equation}\label{eq:excSorted}
		\begin{aligned}
			&j_1 < \pi(j_1) \leq j_2 < \pi(j_2) \leq j_3 < \pi(j_3) \leq \dotsb \leq j_k< \pi(j_k)\text{ and} \\
			&\pi(i_1) \leq i_1 < \pi(i_2) \leq i_2 < \pi(i_3) \leq i_3 <\dotsb < \pi(i_{n-k}) \leq i_{n-k}=n.
		\end{aligned}
	\end{equation}
Moreover, if $i_{n-k-1}<j_1$, then 
\[
i_1 < i_2< \dotsb < i_{n-k-1}<j_1<j_2 < \dotsb < j_k <i_{n-k}
\]
and it follows directly that 
\begin{equation}\label{eq:specialCase}
	\pi = (n-k\quad n-k+1 \quad \dotsc \quad n-1\quad  n)
\end{equation}
with $\inv(\pi)=\exc(\pi)$.
\end{corollary}
\begin{proof}
	The forward statement follows directly from \cref{cor:critical}. Now suppose that \eqref{eq:excSorted} holds. Then 
	\[
	\pi(j_s)\leq j_{s'}\text{ and } i_r< \pi(i_{r'}),\text{ for }s<s'\text{ and }r<r'.
	\] 
	However, after swapping $\pi(j_s)$ and $\pi(j_{s'})$, $j_{s}$ remains an excedance while $j_{s'}$ is not since $\pi(j_{s'})> j_{s'}>j_s$ and $\pi(j_s)\leq j_{s'}$. Similarly, swapping $\pi(i_r)$ and $\pi(i_{r'})$ preserves $i_{r'}$ being an anti-excedance but not $i_{r}$ since $\pi(i_{r})< i_{r}<i_{r'}$ and $\pi(i_{r'})> i_{r}$. Thus, $\pi$ is critical.
\end{proof}

The following is similar to an argument in \cite{Mantaci1993},
where a slightly different\footnote{We note that the original proof has a few typos.}
approach is taken.
\begin{proposition}\label{prop:mantaciInvolution}
	Let $E=\{j_1,j_2,\dotsc,j_k\} \subseteq [n-1]$, and define $\pi_E \in \symS_n$
	with excedance set $E$ via
	\begin{align*}
		\pi_E(j_s) &=j_s+1 \quad &&\text{ for each } j_s\in E, \\
		\pi_E(i_r) &=i_{r-1}+1\quad &&\text{ for each } i_r\in [n]\setminus E, \,\,(i_0 \coloneqq 0).
	\end{align*}
	Then $\pi_E$ is the unique critical permutation in $\symS_n$ with $\EXCi(\pi_E)=E$,
	and $\inv(\pi_E)=|E|$.
\end{proposition}
\begin{proof}
	We first show that $\pi_E$ is critical, so assume it is not. 
	Then swapping $\pi_E(j_{s})$ and $\pi_E(j_{s'})$, for $s < s'$,
	produces a $\pi_E'$ from $\pi_E$ with the same set of (anti)excedances. However, $\pi_E'(j_{s'}) = j_s+1 < j_{s'}+1$ implies $\pi_E'(j_{s'})\leq j_{s'}$. 
	So, the set of excedances is not preserved.
	A similar argument shows that we cannot swap a pair of anti-excedances
	either.
	
	Now we have established that $\pi_E$ is critical, we must show that
	there are no other critical permutations in $\symS_n$ with $E$ as excedance set.
	
	We proceed by (strong) induction over $n$.  The base case $n=1$ is trivial. And if $\pi(n)=n$, then the statement follows easily by induction hypothesis. From now on we consider $\pi^{-1}(n)<n$.
	
	First we handle the case $|E|=n-1$ where $E = \{1,2,\dotsc,n-1\}$.
	There is only one permutation with $E$ as excedance set,
	namely $\pi = (1\;2\; \dotsc \; n)$ in cycle form, and this is exactly $\pi_E$,
	where $\inv(\pi_E)=|E|$.
	
	Suppose now that $|E|=k<n-1$, $E = \{j_1,\dotsc,j_k\}$, and $\{i_1,\dotsc,i_{n-k}\}$ be 
	the set of anti-excedances of some critical permutation $\pi$.
	Let $h \in [n]$ be the largest integer such that $j_h < i_{n-k-1}$.
	If there is no such $h$, then $E = \{n-k,\dotsc,n-2,n-1\}$
	and $\pi$ is of the form given in \eqref{eq:specialCase}
	and the statement holds.
	So now, there is some $m \in [n-k]$ such that $j_h+1=i_{n-k-m}$,
	and the permutation $\pi$ has the following structure:
	\[
	\pi =
	\begingroup
	\setcounter{MaxMatrixCols}{20}
	\setlength\arraycolsep{2pt}
	\scriptstyle{
		\begin{pmatrix}
			1      & 2      & \dotsb & j_h & i_{n-k-m}  & \dotsb   & i_{n-k-1}      & j_{h+1} & \dotsb    & j_{k} &  n \\
			\pi(1) & \pi(2) & \dotsb & \pi(j_h) & \pi(i_{n-k-m}) & \dotsb & \pi(i_{n-k-1}) & \pi(j_{h+1}) & \dotsb & \pi(j_{k})  & \pi(n)
	\end{pmatrix}}.
	\endgroup
	\]
	We have that
	\[
	j_{h+1} < \pi(j_{h+1}), \quad j_{h+2}<\pi(j_{h+2}), \quad j_{k}<\pi(j_{k}),
	\]
	and taking \eqref{eq:excSorted} into account, this is only possible if 
	\[
	\pi(j_{h+1}) = j_{h+2},\quad \pi(j_{h+2}) = j_{h+3},\quad \dotsc, \quad \pi(j_{k}) = n,
	\]
	or $h=k$ and $\pi(j_h)=n$.

	\textbf{Suppose now $\pi(j_h)>i_{n-k-1}$}. Then by \eqref{eq:excSorted},
	either $\pi(j_h)=j_{h+1}$, or $h=k$ and $\pi(j_h)=n$.
	In either case, we must have that
	$\pi(i_{n-k})\leq j_{h+1}-1= i_{n-k-1}$ then $\pi(n)= i_{n-k-1}$ by \eqref{eq:excSorted},
	 and $\pi(i_{n-k-1}) < i_{n-k-1}$.
	But this is not possible, as $\pi$ would not be critical (we could swap
	the values at positions $n-k-1$ and $n$).
	Hence, $\pi(j_h) \leq i_{n-k-1}$.
	\medskip 
	
	It follows that $\pi$ sends $\{1,2,\dotsc,i_{n-k-1}\}$ to itself
	and that $\pi(n) = i_{n-k-1} + 1$. Hence,
	the value of $\pi(s)$ is uniquely determined for all
	$s > i_{n-k-1}$.
	So, $\pi$ restricts to a critical permutation $\pi'$
	acting on $[i_{n-k-1}]$.
	By induction, $\pi'$ is uniquely determined by $E \cap [i_{n-k-1}]$ with 
	so it follows that $\pi$ is unique and of the form $\pi_E$.
	Also by induction, $\inv(\pi') = h = |E \cap [i_{n-k-1}]|$,
	and finally 
	\[
	\inv(\pi) = \inv(\pi') + (k-h) = h + (k-h) = |E|.
	\]
\end{proof}

\begin{example}
 Let $n=4$ and consider all permutations with $2$ excedances.
 We have $7$ even permutations with two excedances,
 and $4$ odd permutations. The sign-reversing involution should therefore have 
 $\binom{3}{2}=3$ fixed-points, all with even sign $(-1)^2$.
 \begin{center}
\begin{tabular}{cc}
Even & Odd \\
\toprule
$\mathbf{1342}$ & \\
$\mathbf{2143}$ & \\
$\mathbf{2314}$ & \\
2431 & 2413  \\
3241 & 3142 \\
3412 & 3421 \\
4321 & 4312 \\
\end{tabular}
\end{center}
\end{example}

\cref{prop:mantaciInvolution} now immediately gives a bijective proof of the following result, which is essentially due to R.Mantaci in \cite{Mantaci1993} (albeit stated in terms of anti-excedances instead of excedances).
\begin{proposition}\label{lem:excEvenOdd}
Let $n \geq 1$, then
\begin{equation}\label{eq:excEvenOdd}
 \sum_{\substack{\pi \in \symS_n}} (-1)^{\inv(\pi)} \xvec_{\EXCi(\pi)} = 
 \prod_{j \in [n-1]} (1-x_j) = \sum_{E \subseteq [n-1]} (-1)^{|E|}\xvec_{E}.
\end{equation}
In particular, by setting all $x_i$ equal to $t$, 
we have
\[
 \sum_{\substack{\pi \in \symS^e_n}} t^{\exc(\pi)} 
 -
 \sum_{\substack{\pi \in \symS^o_n }} t^{\exc(\pi)} = (1-t)^{n-1}.
\]
\end{proposition}

\begin{proposition}\label{prop:main}
Let $n\geq 1$ and let $T \subseteq [n]$. 
Let $m \leq n$ be the largest integer not in $T$ and set
$
  E = \{1,2,\dotsc,m-1\} \setminus T.
$
Then
 \begin{equation}\label{eq:main}
 \sum_{\substack{\pi \in \symS_n  \\ T \subseteq \FIX(\pi)  }} 
 (-1)^{\inv(\pi)} \xvec_{\EXCi(\pi)} =  
 \prod_{j \in E} (1-x_j),
\end{equation}
where the empty product has value $1$.

Setting all $x_i$ to be $t$, we have
\begin{equation}
 \sum_{\substack{\pi \in \symS^e_n  \\ T \subseteq \FIX(\pi)  }} t^{\exc(\pi)}
 -
 \sum_{\substack{\pi \in \symS^o_n \\ T \subseteq \FIX(\pi)   }} t^{\exc(\pi)}
 = 
 \begin{cases}
 1 & \text{if $|T|=n$} \\
 (1-t)^{n-1-|T|} & \text{otherwise.}
 \end{cases}
\end{equation}
\end{proposition}
\begin{proof}
First note that $E = \emptyset$ if $T=[n]$ and \eqref{eq:main} is easy to verify,
so from now on, we may assume $|T| < n$.

By definition of $m$, we have that $T = T_1 \cup T_2$
where $T_1 \subseteq \{1,2,\dotsc,m-1\}$, and $T_2 = \{m+1,m+2,\dotsc,n\}$.
Hence, $|E|+|T_1| = m-1$ and $|T_2| = n-m$, and 
\[
|E| = n-1-|T_1|-|T_2| = n-1-|T|.
\]
Now suppose $\pi \in \symS_n$ is a permutation such that $T \subseteq \FIX(\pi)$.
We then construct $\pi' \in \symS_{n-|T|}$, by only considering the positions not in $T$,
and the relative ordering of the entries at these positions.
For example, if $n=9$, $T=\{2,4,6,8,9\}$,
$[n] \setminus T = \{1,3,5,7\}$ and
\[
  \pi = \underline{1} 2 \underline{7} 4 \underline{3} 6 \underline{5} 8 9,
  \text{ we have } \pi' = 1423
\]
since the relative ordering of $1,3,5$ and $7$ in $\pi$ is $1423$.
Observe that $\exc(\pi) = \exc(\pi')$ and $(-1)^{\inv(\pi)} = (-1)^{\inv(\pi')}$.

Hence, the sum in the left-hand side of \eqref{eq:main},
can be taken as a sum over permutations $\pi' \in \symS_{n-|T|}$,
but with a reindexing of the variables using values in $[n] \setminus T$.
Now, this sum can be computed using \cref{lem:excEvenOdd}
which finally gives \eqref{eq:main}.
Note that $m$ is the largest member of $[n] \setminus T$, 
so we do not get any variable with this index --- this corresponds
to the fact that the right-hand side of \eqref{eq:excEvenOdd}
only uses elements in $[n-1]$.
\end{proof}

\begin{theorem}\label{thm:mainGenFunc}
Let $n\geq 1$. Then
\begin{equation}\label{eq:derangementSum}
 \sum_{\pi \in \dArr_n } (-1)^{\inv(\pi)} \xvec_{\EXCi(\pi)} = 
 (-1)^{n-1} \sum_{j=1}^{n-1} x_{1}x_{2}\dotsm x_{j}.
\end{equation}
\end{theorem}
\begin{proof}
By inclusion-exclusion, we have the two identities:
\begin{align*}
 \sum_{\substack{\pi \in \symS^e_n \\ \FIX(\pi)=\emptyset}} \xvec_{\EXCi(\pi)}
 &=
 \sum_{\substack{T \subseteq [n]  }}
 (-1)^{|T|} \sum_{\substack{\pi \in  \symS^e_n  \\ T \subseteq \FIX(\pi)   }} \xvec_{\EXCi(\pi)},\\
 \sum_{\substack{\pi \in \symS^o_n \\ \FIX(\pi)=\emptyset}} \xvec_{\EXCi(\pi)}
 &=
 \sum_{\substack{T \subseteq [n]   }}
 (-1)^{|T|} \sum_{\substack{\pi \in  \symS^o_n  \\ T \subseteq \FIX(\pi)   }} \xvec_{\EXCi(\pi)}.
\end{align*}
By taking the difference of these two identities, we get
\begin{align*}
 \sum_{\substack{\pi \in \dArr_n }} (-1)^{\inv(\pi)} \xvec_{\EXCi(\pi)}
 &=
\sum_{\substack{T \subseteq [n] }}
 (-1)^{|T|}
\left( 
 \sum_{\substack{\pi \in  \symS^e_n  \\ T \subseteq \FIX(\pi)   }} \xvec_{\EXCi(\pi)} 
 -
 \sum_{\substack{\pi \in  \symS^o_n  \\ T \subseteq \FIX(\pi)   }} \xvec_{\EXCi(\pi)}
\right). 
\end{align*}
By \cref{prop:main}, the difference in the right-hand side is equal to 
\[
  \prod_{ j \in [m_T-1]\setminus T } (1-x_j)
\]
where $m_T\leq n$ is the largest integer not in $T$.
We group the terms depending on the value of $m_T$.
If $m_T=0$ then $T=[n]$ and the product is empty, so its value is $1$.
In general, the left hand side of \eqref{eq:derangementSum} is equal to
\[
(-1)^{n} + 
 \sum_{k=1}^n \sum_{\substack{T \subseteq [n] \\ m_T = k}} (-1)^{|T|}
 \prod_{ j \in [k-1]\setminus T } (1-x_j). 
 \]
By using $E = [k-1]\setminus T$, this can then be expressed as
 \[
 (-1)^{n} +
 \sum_{k=1}^n 
 \sum_{\substack{E \subseteq [k-1]}} (-1)^{n-1-|E|}
 \prod_{ j \in E } (1-x_j).
\]
Canceling the $k=1$ case with $(-1)^{n}$, and then shifting the index, we get
 \[
(-1)^{n-1}
 \sum_{k=1}^{n-1} 
 \sum_{\substack{E \subseteq [k]}} 
 \prod_{ j \in E } (x_j-1).
\]
Now,
\begin{align*}
(-1)^{n-1}
\sum_{k=1}^{n-1} 
\sum_{\substack{E \subseteq [k]}} 
\prod_{ j \in E } (x_j-1)
&=
(-1)^{n-1}
\sum_{k=1}^{n-1} 
\sum_{\substack{E \subseteq [k]}} 
\sum_{F \subseteq E} (-1)^{|E|-|F|}\xvec_F \\
&=
(-1)^{n-1}
\sum_{k=1}^{n-1} 
\sum_{\substack{F \subseteq [k]}} 
(-1)^{|F|}\xvec_F
\sum_{F \subseteq E \subseteq [k]} (-1)^{|E|}.
\end{align*}
The last sum vanish unless $F=[k]$, and we have that 
\begin{equation}
(-1)^{n-1}
\sum_{k=1}^{n-1} 
\sum_{\substack{E \subseteq [k]}} 
\prod_{ j \in E } (x_j-1)
=
(-1)^{n-1}\sum_{k=1}^{n-1} \xvec_{[k]},
\end{equation}
which is exactly the right-hand side in \eqref{eq:derangementSum}.
\end{proof}

\begin{corollary}
For $n,k \geq 1$, we have that
\[
|\{ \pi \in \dArr^e_n : \exc(\pi) = k \}|-|\{ \pi \in \dArr^o_n : \exc(\pi) = k \}| = (-1)^{n-1}.
\]
\end{corollary}
\begin{proof}
 This follows directly by comparing
 coefficients of degree $k$ in \eqref{eq:derangementSum}.
\end{proof}

\subsection{A right-to-left minima analog}

Our goal with this subsection is to prove 
a right-to-left minima analog of \cref{lem:excEvenOdd}.
We shall use the same type of proof, i.e., by exhibiting an involution on $\symS_n$,
such that all fixed-elements with the same set of right-to-left minima,
also have the same sign.

\begin{definition}
Let $\kappa : \symS_n \to \symS_n$ be defined as follows.
Given $\pi \in \symS_n$, let $i \in [n]$ be the smallest \emph{odd} integer 
such that $\pi (i \; i+1)$ and $\pi$ have the same sets of right-to-left minima,
if such an $i$ exists. That is, we swap the entries at positions $i$ and $i+1$ in $\pi$.
We then set $\defin{\kappa(\pi)} \coloneqq \pi (i \; i+1)$,
and $\kappa(\pi) \coloneqq \pi$ otherwise.
We say that $\pi$ is \defin{decisive}\footnote{As a nod to the word \emph{critical}.}
if it is a fixed-point of $\kappa$.
\end{definition}

\begin{example}
In $\symS_7$, there are $8$ decisive permutations:
\[
1234567,\;
1234657,\;
1243567,\;
1243657,\;
2134567,\;
2134657,\;
2143567,\;
2143657.
\]
Note that $\{1,3,5,7\}$ are always right-to-left minima (but there might be more).
\end{example}

\begin{lemma}\label{lem:ltrMinEvenOdd}
The map $\kappa : \symS_n \to \symS_n$ has the following properties:
\begin{enumerate}[label=({\roman*})]
 \item \label{it:kappaprop1} $\kappa$ is an involution,
 \item \label{it:kappaprop2} $\kappa$ preserves the number of right-to-left minima,
 \item \label{it:kappaprop3} $\kappa$ changes sign of non-fixed elements,
\item \label{it:kappaprop4} 
 For each subset $T \in [n] \cap \{2,4,6,\dotsc\}$,
 there is a unique decisive permutation with 
 $\{1,3,5,\dotsc\} \cup T$ as right-to-left minima set.
  
 \item \label{it:kappaprop5} 
 there are $\binom{ \lfloor n/2 \rfloor }{k - \lceil n/2 \rceil}$
 decisive permutations with exactly $k$ right-to-left minima,
 and they all have sign $(-1)^{n-k}$.
 \end{enumerate}
\end{lemma}
\begin{proof}
Items \ref{it:kappaprop1}--\ref{it:kappaprop3} are clear from the definition of $\kappa$.
It remains to prove \ref{it:kappaprop4} and \ref{it:kappaprop5}.
Let us use $O$ to denote the odd integers in $[n]$, and let $E$
be the even integers in $[n]$.
In order to prove \ref{it:kappaprop4}, we must construct a decisive permutation $\pi$,
such that $\RLMv(\pi) = O \cup T$. We construct $\pi$ from $T$
according to the following rules:
\begin{itemize}
\item if $n$ is odd, then $\pi(n)=n$.
\item if $j \in O$, $j<n$ we have that
\[
\begin{cases}
\pi(j) =j \text{ and } \pi(j+1) = j+1 &\text{ if } j+1 \in T \\
\pi(j) =j+1 \text{ and } \pi(j+1) = j &\text{ if } j+1 \notin T.
\end{cases}
\]
\end{itemize}
In short, $\pi$ is constructed by first placing $1$ and $2$,
in the order determined by $T$, then $3$ and $4$, etc.
By construction, $\RLMv(\pi) = O \cup T$.
Now we must show that a permutation is decisive if and only if it is of this form. From the construction, it is clear that
$\pi (i\; i+1)$ and $\pi$ do not have the same set of right-to-left minima,
for any choice of $i \in O$. Hence, all permutations with this structure are decisive.
\medskip

\textbf{Claim:} \emph{Every decisive permutation has the structure described above}.

First note that the claim is true for $n=1$ and $n=2$, so suppose $n\geq 3$.
Now, if $n$ does not appear among the last two entries of $\pi$,
then $n$ is not a right-to-left minima. Moreover, we can swap $n$ 
with the entry either to its right or to its left, and preserve the 
set of right-to-left minima. 
In particular, if $n$ does not appear among the last two positions,
then $\pi$ is not decisive. Now, if $\pi(n)=n$, we can remove 
the last entry and use induction.
Otherwise, suppose $\pi(n-1)=n$ and $\pi(n)<n-1$. In particular,
$n-1 \notin \RLMv(\pi)$.
It is then possible to swap $n-1$ with one of its neighbors and preserve $\RLMv(\pi)$ in a manner, which shows that $\pi$ is not decisive.  We conclude that if 
$\pi(n) \neq n$, then $\pi(n-1)=n$ and $\pi(n)=n-1$ in order for $\pi$
to be decisive. Now, we may remove the last two entries,
and proceed by induction. This ends the proof of the claim.
\smallskip

From \ref{it:kappaprop4}, we know that in order to 
construct a decisive permutations in $\symS_n$ with $k$ right-to-left minima,
we must include all $\lceil n/2 \rceil$ odd integers in $[n]$,
and pick a subset of size $k  - \lceil n/2 \rceil$,
from the set of even integers in $[n]$.
The subset of even integers has cardinality $\lfloor n/2 \rfloor$.
Hence, we get the advertised formula in 
\ref{it:kappaprop5}, so it suffices to show that all such
decisive permutations have the same sign.
By the claim above, it is evident that 
the sign only depends on the number of right-to-left minima in $\pi$,
and from here, it is straightforward for show that it is indeed
$(-1)^{n-k}$.
\end{proof}

\begin{corollary}\label{cor:rlminsum}
We have that for any $n\geq 1$
\begin{equation}\label{eq:rlminSymSum}
 \sum_{\substack{\pi \in \symS_n   }} (-1)^{\inv(\pi)} \xvec_{\RLMv(\pi)}
 = 
	\big(\prod_{\substack{i \in [n] \\ i \text{ odd}}} x_i\big)
	\big(\prod_{\substack{j \in [n] \\ j \text{ even}}} (x_j-1)\big).
\end{equation}
In particular, for any $k = 1,\dotsc,n$ we have that
\[
 |\{ \pi \in \symS^e_n : \rlm(\pi)=k \}| - 
 |\{ \pi \in \symS^o_n : \rlm(\pi)=k \}| = (-1)^{n-k}
 \binom{ \lfloor n/2 \rfloor }{k - \lceil n/2 \rceil}.
\]
\end{corollary}
\begin{proof}
This follows directly from \cref{lem:ltrMinEvenOdd},
where the first product in \eqref{eq:rlminSymSum}
corresponds to the fact that all odd integers in $[n]$ must be 
right-to-left minima for decisive permutations,
and the second product corresponds to choosing a
subset $T$ among the even numbers in $[n]$.
It remains to check that the signs are chosen correctly,
which is straightforward as well.

The second statement also follows from \cref{lem:ltrMinEvenOdd},
or by simply comparing coefficients of $t^k$ in \eqref{eq:rlminSymSum},
after letting $x_j \to t$.
\end{proof}

We conclude this section with the following problem.
\begin{problem}
Is it possible to state an analog of \cref{prop:main}?
In particular, for $T \subseteq [n]$, is there a nice expression for the sum
\[
 \sum_{\substack{\pi \in \symS_n  \\ T \subseteq \FIX(\pi) }}(-1)^{\inv(\pi)} t^{\rlm(\pi)}  ?
\]
Computer experiments suggest that this sum is either $0$ or of the form 
$\pm t^a (t+1)^b (t-1)^c$, where $a$, $b$, and $c$ depend on $T$ in some manner.
\end{problem}

\section{Further ideas and conjectures}

\subsection{Multiderangements}

Let $\defin{B_n} \coloneqq (1,1,2,2,3,3,\dotsc,n,n)$ be fixed.
A \defin{biderangement} of $B_n$, is a permutation, $\wvec$, of the entries in $B_n$, 
such that $\wvec(j) \neq B_n(j)$ for all $j \in [2n]$.
The set of biderangements of $B_n$ is denoted $\defin{\bdArr_n}$.
The cardinality of $\bdArr_n$ is given by \oeis{A000459},
which starts as $0, 1, 10, 297, 13756,\dotsc$.

We compute the number of \defin{inversions} 
of a biderangements as for words in general.
Moreover, we say that $\wvec(j)$ is an \defin{excedance value} of $\wvec \in \bdArr_n$
if $\wvec(j) > B_n(j)$. This defines the multi-set valued statistic $\EXCv(\wvec)$.
We also define the \defin{right-to-left minima values}
as the set 
\[
   \RLMv(\wvec) \coloneqq \{ \wvec_i : \wvec(i) < \wvec(j) \text{ for all $j \in \{i,i+1,\dotsc,2n\}$} \}.
\]

\begin{example}
In the table below, we show the ten elements in $\bdArr_3$,
together with the corresponding inversion and excedance statistics.
\begin{center}
\begin{tabular}{ccccccccc}
 \text{Biderangement} & \text{inv}  & \text{EXCv} & \text{RLMv} & $\quad$ & \text{Biderangement}  & \text{inv} & \text{EXCv} & \text{RLMv} \\
   \toprule \\
 223311&8&\{2,2,3,3\}&\{1\}   &&  231312&7&\{2,3,3\}&\{1,2\}  \\
 231321&8&\{2,3,3\}&\{1\}   &&  233112&8&\{2,3,3\}&\{1,2\}  \\
 233121&9&\{2,3,3\}&\{1\}  &&  321312&8&\{2,3,3\}&\{1,2\}  \\
 321321&9&\{2,3,3\}&\{1\}   &&  323112&9&\{2,3,3\}&\{1,2\}  \\
 323121&10&\{2,3,3\}&\{1\}  &&  331122&8&\{3,3\}&\{1,2\} 
\end{tabular}
\end{center}
\end{example}

\begin{proposition}
For $n\geq 1$ , we have that 
\begin{equation}\label{bidera}
 \sum_{\wvec \in \bdArr_n} (-1)^{\inv(\wvec)} 
 \xvec_{\EXCv(\wvec)} \yvec_{\RLMv(\wvec)}
 = 
 \sum_{\pi \in \dArr_n} \xvec_{\EXCv^2(\pi)} \yvec_{\RLMv(\pi)},
\end{equation}
where $\EXCv^2(\pi)$ is the multiset obtained from $\EXCv(\pi)$ by repeating each element twice.
\end{proposition}
\begin{proof}
Define a mapping $\beta:\bdArr_n\to\bdArr_n$ as $\beta(\wvec)=\wvec'$, where $\wvec'$ is obtained
from $\wvec$ by switching $\wvec(j)$ and $\wvec(j+1)$ for the 
smallest \emph{odd} $j\in [2n-1]$ such that $\wvec(j) \neq \wvec(j+1)$.
If there is no such $j$, then $\beta(\wvec)\coloneqq\wvec$.
Since $\wvec'(j)=\wvec(j+1) \neq B_n(j+1) = B_n(j)$ and $\wvec'(j+1) = \wvec(j) \neq B_n(j)=B_n(j+1)$, 
$\beta$ is a sign-reversing involution which preserves
the excedance set values; $\EXCv(\wvec)=\EXCv(\wvec')$.
Now, each number appear twice in a biderangement,
so $\wvec(j)$ and $\wvec(j+1)$ each appear again somewhere to the right 
of $j$ and $j+1$, respectively.
Hence, the positions $j$ and $j+1$ in $\wvec$ and $\wvec'$
cannot be right-to-left minima indices, and it follows that 
$\RLMv(\wvec)=\RLMv(\wvec')$.

The elements fixed by $\beta$ are bijection with the derangements;
we send $\pi \in \dArr_n$ to the biderangement
$\vvec=\pi(1)\pi(1)\pi(2)\pi(2) \dotsb \pi(n)\pi(n)$.
Consequently, $\EXCv(\vvec)$ is the multiset obtained from 
$\EXCv(\pi)$ by repeating each element twice, 
and $\RLMv(\vvec)=\RLMv(\pi)$. This concludes the proof of \eqref{bidera}.
\end{proof}

\subsection{A generalization}

The set of dearrangements are permutations where cycles of length $1$
are disallowed. With this in mind, it is reasonable to explore what 
happens if we add restrictions to the length of the cycles.
Recall that the \defin{type}, \defin{$\type(\pi)$} of a permutation is the 
integer partition $(\mu_1,\mu_2,\dotsc,\mu_\ell)$
where the parts are the cycle lengths of $\pi$ arranged in decreasing order.

\begin{conjecture}
Let $k$ be a fixed positive integer such that $1\leq k\leq n$.
Then
\[
 \sum_{\substack{\pi \in \symS_n\\ \min(\type(\pi))\geq k}} (-1)^{\inv(\pi)}
 \xvec_{\RLMv(\pi)} \yvec_{ \EXCv(\pi)},
\]
where $\type(\pi)$ is the cycle type of $\pi$, is an element in $\setN[x_1,\dotsc,x_n,y_1,\dotsc,y_n]$.
\end{conjecture}

The case $k=2$ follows from \cref{thm:theMainTheorem}.
For the case $k=n$, all permutations have the same sign, so the statement is trivial.
Interestingly, summing over all permutations consisting of a single cycle of length $n$,
\begin{equation}\label{eq:EulerianSum}
 \sum_{\substack{\pi \in \symS_n \\ \type(\pi) = (n)}}
 \xvec_{\RLMv(\pi)} \yvec_{ \EXCv(\pi)}
\end{equation}
is a multivariate polynomial where the number of terms (not counting multiplicity!)
seems to be given by the sequence \oeis{A124302}, which starts as
\[
 1, 1, 2, 5, 14, 41, 122, 365, 1094, \dotsc.
\]
For $n=3$, \eqref{eq:EulerianSum} is equal to
\[
 x_1 x_2 x_3 y_4 + 
 x_1 x_3 y_2 y_4 + 
 x_1 y_3 y_4 + 
 2 x_1 x_2 y_3 y_4 +
 x_1 y_2 y_3 y_4
\]
which has $5$ terms.

\subsection{Right to left minima and derangements}

Let $a_{n,k}$ be defined via
\begin{equation}
 \sum_{\pi \in \dArr_n} t^{\rlm(\pi)} = \sum_{k=1}^n a_{n,k} t^k,
\end{equation}
so that $a_{n,k}$ is the number of derangement with exactly $k$ right-to-left minima.
For example, the data for $a_{n,k}$ is shown in the table below.

\begin{table}[!ht]
\begin{tabular}{rrrrrrrr}
 & \textbf{1} & \textbf{2} & \textbf{3} & \textbf{4}& \textbf{5}& \textbf{6}& \textbf{7} \\
\toprule
\textbf{2}& 1 \\
\textbf{3}&1 & 1\\
\textbf{4}&3 & 5 & 1 \\
\textbf{5}&11 & 21 & 11 & 1 \\
\textbf{6}&53 & 113 & 79 & 19 & 1 \\
\textbf{7}&309 & 715 & 589 & 211 & 29 & 1  \\
\textbf{8}&2119 & 5235 & 4835 & 2141 & 461 & 41 & 1 \\
\end{tabular}
\caption{The number of derangements of $[n]$ with exactly $k$ right-to-left minima.}\label{tab:deArrRLM}
\end{table}

It is straightforward from the definition to see that $a_{n,1}$ is the number of derangement ending with a $1$.
The sequence $a_{n,1}$ shows up in the OEIS as \oeis{A000255}, which hints the recursion
\[
 a_{n,1} = (n-2) \cdot a_{n-1,1} + (n-3) \cdot a_{n-2,1}, \qquad a_{1,1} =1, \qquad a_{2,1} =1.
\]
Moreover, it seems that $a_{n,n-1} = (n-2) + (n-1)^2$.
The data in \cref{tab:deArrRLM} is not in the OEIS, 
and we leave it as an open problem to describe the 
entries via recursions or closed-form formulas. 

A straightforward recursion for the number of elements in $\dArr_n$ with $k$ excedances,
can be found in \cite{MantaciRakotondrajao2001}.

\subsection*{Acknowledgements}

The authors have benefited greatly from the discussions
with Jörgen Backelin of Stockholm University.
The Online Encyclopedia of Integer sequences \cite{OEIS} 
has been an invaluable resource in our project.

The second author is grateful for the financial support extended
by the cooperation agreement between International Science Program
at Uppsala University and Addis Ababa University.

\bibliographystyle{alphaurl}

\begin{thebibliography}{Man93}

\bibitem[BV17]{BarilVajnovszki2017}
Jean-Luc Baril and Vincent Vajnovszki.
\newblock A permutation code preserving a double {E}ulerian bistatistic.
\newblock {\em Discrete Applied Mathematics}, 224:9--15, June 2017.
\newblock \href {https://doi.org/10.1016/j.dam.2017.02.014}
  {\path{doi:10.1016/j.dam.2017.02.014}}.

\bibitem[Cha01]{Chapman2001}
Robin Chapman.
\newblock An involution on derangements.
\newblock {\em Discrete Mathematics}, 231(1-3):121--122, March 2001.
\newblock \href {https://doi.org/10.1016/s0012-365x(00)00310-1}
  {\path{doi:10.1016/s0012-365x(00)00310-1}}.

\bibitem[Man93]{Mantaci1993}
Roberto Mantaci.
\newblock Binomial coefficients and anti-exceedances of even permutations: A
  combinatorial proof.
\newblock {\em Journal of Combinatorial Theory, Series A}, 63(2):330--337, July
  1993.
\newblock \href {https://doi.org/10.1016/0097-3165(93)90064-f}
  {\path{doi:10.1016/0097-3165(93)90064-f}}.

\bibitem[MR01]{MantaciRakotondrajao2001}
Roberto Mantaci and Fanja Rakotondrajao.
\newblock A permutations representation that knows what ``{E}ulerian'' means.
\newblock {\em {Discrete Mathematics \& Theoretical Computer Science}}, 4(2),
  January 2001.
\newblock URL: \url{https://dmtcs.episciences.org/271}.

\bibitem[MR03]{MantaciRakotondrajao2003}
Roberto Mantaci and Fanja Rakotondrajao.
\newblock Exceedingly deranging!
\newblock {\em Advances in Applied Mathematics}, 30(1-2):177--188, February
  2003.
\newblock \href {https://doi.org/10.1016/s0196-8858(02)00531-6}
  {\path{doi:10.1016/s0196-8858(02)00531-6}}.

\bibitem[Slo19]{OEIS}
Neil J.~A. Sloane.
\newblock The {O}n-{L}ine {E}ncyclopedia of {I}nteger {S}equences.
\newblock Online, 2019.
\newblock URL: \url{https://oeis.org}.


\end{thebibliography}

\end{document}